\documentclass[reqno]{amsart}
\usepackage[margin=3cm]{geometry}
\usepackage{amsmath}
\usepackage{amssymb}
\usepackage{amsbsy}
\usepackage[latin1]{inputenc}
\usepackage{dsfont}
\usepackage{empheq}

\usepackage{siunitx}
\usepackage{wasysym} 
\usepackage{amssymb}

\def\a {{\boldsymbol a}}

\def\e {{\boldsymbol e}}

\def\u {{\boldsymbol u}}
\def\n {{\boldsymbol n}}
\def\x {{\boldsymbol x}}
\def\v {{\boldsymbol v}}
\def\w {{\boldsymbol w}}
\def\p {{\boldsymbol p}}

\def\H{{\boldsymbol H}}

\def\X{{\boldsymbol  X}}

\def\H{{\boldsymbol H}}
\def\U{{\boldsymbol U}}

\def\X{{\boldsymbol  X}}

\def\R {\mathds{R}}
\def\S {\mathbb{S}}
\newcommand{\meas}[1]{{\it meas}(#1)}

\newtheorem{theorem}{Theorem}[section]
\newtheorem{lemma}[theorem]{Lemma}
\newtheorem{proposition}[theorem]{Proposition}
\newtheorem{corollary}[theorem]{Corollary}

\newtheorem{remark}[theorem]{Remark}

\begin{document}

\title[Inf-sup condition for the LLG and harmonic map heat equation]{Inf-sup stable finite-element methods for the Landau--Lifshitz--Gilbert and harmonic map heat flow equation}

\author{Juan Vicente Guti\'errez-Santacreu$^\dag$}
\author{Marco Restelli$^\ddag$}

\thanks{$\dag$ Dpto. de Matem\'atica Aplicada I, Universidad de Sevilla, E.
T. S. I. Inform\'atica. Avda. Reina Mercedes, s/n. 41012 Sevilla,
Spain. {\tt juanvi@us.es}. Partially supported by  Ministerio de Economía y Competitividad under Spanish grant MTM2015-69875-P with the participation of FEDER} 

\thanks{$\ddag$ Numerische Methoden in der Plasmaphysik, Max-Planck-Institut f\"ur Plasmaphysik, Boltzmannstr. 2, 85748 Garching, Germany ({\tt marco.restelli@ipp.mpg.de}). }

\date{\today}
\begin{abstract} In this paper we propose and analyze a finite element method for both the harmonic map heat and Landau--Lifshitz--Gilbert equation, the time variable remaining continuous. Our starting point is to set out a unified saddle point approach for both problems in order to impose the unit sphere constraint at the nodes since the only polynomial function satisfying the unit sphere constraint everywhere are constants. A proper inf-sup condition is proved for the Lagrange multiplier leading to the well-posedness of the unified formulation. \emph{A priori} energy estimates are shown for the proposed method.    

When time integrations are combined with the saddle point finite element approximation some extra elaborations are required in order to ensure both \emph{a priori} energy estimates for the director or magnetization vector depending on the model and an inf-sup condition for the Lagrange multiplier. This is due to the fact that the unit length at the nodes is not satisfied in general when a time integration is performed.  We will carry out a linear Euler time-stepping method and a non-linear Crank--Nicolson method. The latter is solved by using the former as a non-linear solver.

\end{abstract}
\maketitle


\tableofcontents

\section{Introduction}
\subsection{The model} 
In this paper we propose and analyze inf-sup stable finite element approximations for the harmonic map heat and Landau--Lifshitz--Gilbert equation.  The unified equations are given by 
\begin{equation}\label{HFHM}
	\left\{
		\begin{array}{rccl}
        			\partial_t \u - \gamma\Delta\u- \gamma|\nabla\u|^2 \u + \alpha\u\times\partial_t\u 
			&=&\boldsymbol{0} &\mbox{ in } \Omega\times\R^+,
\\
        			|\u|&=& 1&\mbox{ in } \Omega\times\R^+,
\\
			\partial_\n\u&=&\boldsymbol{0}&\mbox{ on }\partial\Omega\times \R^+,
\\
			\u(0)&=&\u_0&\mbox{ in }\Omega,

\end{array}
\right.
\end{equation}
where $\u:\Omega\times \R^+\to \S^{M-1}$, with $M=2$ or $3$ being the space dimension, $\Omega$ is a bounded domain of $\R^M$ with boundary $\partial\Omega$, $\S^{M-1}$ is the unit $(M-1)$-sphere, $\partial_\n$ is the normal derivative, with $\n$ being the unit outward normal vector on $\partial\Omega$, and $\gamma,\alpha\in\R$ with$\gamma>0$ and $\alpha\geq0$; for $M=2$ we assume that $\alpha=0$, so that the last term in the right-hand-side of~(\ref{HFHM})$_1$ appears only in the three-dimensional case. The normalization condition (\ref{HFHM})$_2$, where $|\cdot|$ stands for the Euclidean norm for vectors and matrices and which will be referred to as \emph{the unit sphere constraint}, is assumed to be satisfied by the initial condition $\u_0$, i.e.\ $|\u_0|=1$; it can be verified that this assumption, together with (\ref{HFHM})$_1$, implies (\ref{HFHM})$_2$ for every time $t>0$.

Equations (\ref{HFHM}) arise in the phenomenological description of widely different physical systems. According to the theory developed by Ericksen \cite{Ericksen_1961, Ericksen_1987} and Leslie \cite{Leslie_1968, Leslie_1979}, system (\ref{HFHM}) for $\alpha=0$ may govern the dynamics of a nematic crystal fluid in the limit of low fluid velocity, where the coupling to the fluid motion is negligible. Here, $\u$ represents the orientation of the liquid crystal molecules, which are modeled as elongated rods tending to line up locally along a preferred direction, while $\gamma$ stands for a relaxation time constant. According to the theory by Landau and Lifshitz  \cite{Landau-Lifshitz} and Gilbert \cite{Gilbert} (in his apparently unpublished first version in \cite{Gilbert_1955}), system (\ref{HFHM}) may also govern the dynamics of magnetization in ferromagnetic materials in the classical continuum approximation, where the relativistic interactions are modeled by the damping term $\alpha\u\times\partial_t\u$ and the thermal fluctuations are negligible. To be more precise, for such a case, the original equation takes the form 
\begin{equation}\label{LLG}
\partial_t\u+\gamma\u\times(\u\times\Delta\u)+\alpha\u\times\partial_t\u=\boldsymbol{0},
\end{equation}
which can be recast as (\ref{HFHM}) by using the identity
\begin{equation}\label{relation}
\u\times(\u\times\Delta\u)=-\Delta\u-|\nabla\u|^2\u.
\end{equation} 
Here, $\u$ stands for the magnetization vector without the presence of an applied magnetic field, while $\gamma$ and $\alpha$ stand for the electron gyromagnetic radius and a damping parameter, respectively. 

In constructing a numerical algorithm for approximating (\ref{HFHM}), one looks for an energy law which is satisfied at the continuous level; such an energy law can be obtained as follows.  Multiplying $(\ref{HFHM})_1$ by $\partial_t\u$ and integrating over $\Omega$, we obtain
$$
\int_\Omega |\partial_t\u(\x)|^2 \,{\rm d}\x- \gamma\int_\Omega \Delta\u(\x)\cdot\partial_t\u(\x) \,{\rm d}\x-\gamma\int_\Omega|\nabla\u(\x)|^2\u(\x)\cdot\partial_t \u(\x) \,{\rm d}\x=0.
$$
Since $\u\cdot\partial_t \u = 0$ by virtue of $\eqref{HFHM}_2$, the third term in this expression vanishes, while the second one can be rewritten as
$$
-\int_\Omega \gamma \Delta\u(\x)\cdot\partial_t\u(\x) \,{\rm d}\x=\frac{1}{2}\frac{\rm d}{ {\rm d}t}\int_\Omega\gamma|\nabla\u(\x)|^2\, {\rm d}\x
$$
by using the Green formula and the homogeneous boundary condition $\eqref{HFHM}_3$. Thus, the energy law for (\ref{HFHM}) reads
\begin{equation}\label{energy}
\int_\Omega|\partial_t\u(\x)|^2 {\rm d}\x+\frac{1}{2}\frac{\rm d}{ {\rm d}t} \int_\Omega \gamma |\nabla\u(\x)|^2 {\rm d}\x=0.
\end{equation}

The fact that, in the derivation of~(\ref{energy}), the unit sphere
constraint is invoked in a pointwise sense has important implications at
the time of deriving a numerical discretization for~(\ref{HFHM}), where
it turns out being a major source of difficulties. In fact, two
contradicting requirements must be accounted for: on the one hand, using
standard piecewise polynomial finite element spaces, the only
possibility to satisfy the unit sphere constraint in a pointwise sense,
which would then allow repeating the derivation of~(\ref{energy}) also
at the discrete level, is using piecewise constant functions; on the
other hand, (\ref{HFHM})$_1$ calls for more regularity in the
approximating space for $\u$ than is provided by a piecewise constant
function.

To this end, we note that various approaches have been considered in the literature. 

One first possibility was considering a projection step which, in its rudimental version \cite{Prohl_2001}, consisted in enforcing the unit sphere constraint at the sole finite--element nodes. This resulted in a numerical scheme using first order conforming finite elements which did not enjoying a discrete energy law.  In~\cite{Alouges-Jaisson}, a refined version of the method was proposed, where a finite--element approximation of $\partial_t\u$ was computed in a suitable tangent space and for which convergence to weak solutions could be proved under the assumption that the space and time discretization parameters tend to zero in a specified way. Such a restriction on the space and time discretization parameters was motivated by the use of an explicit first-order time integrator; \cite{Alouges_2008} then introduced a formulation which circumvented this drawback using a $\theta$-method. In this latter formulation, for $\theta\in(\frac{1}{2},1]$, the algorithm was unconditionally energy stable and convergent. Yet, the main limitation is that the projection step prevented the scheme from being second order accurate in time; subsequent modifications addressing this have been considered in~\cite{Alouges_Kristsikis_Toussaint,
Alouges-Kritsikis-Steiner-Toussaint}. 

One second possibility was using closed nodal integration together with reformulation (\ref{LLG}), in order to avoid the projection step, required to enforce the nodal fulfillment of the unit sphere constraint.  This approach has been successfully used with a Crank--Nicolson time integration to obtain numerical schemes which satisfied a discrete energy law and preserved the unit sphere constraint at the nodes while converging toward weak solutions. The conditional solvability of this approach is the main disadvantage with respect to the projection method. We refer to \cite{Bartels_Prohl_2006} for the Landau--Lifshitz--Gilbert equation and \cite{Bartels_Prohl_2007} for the
harmonic map heat flow equation. 

One third option~\cite{Prohl_2001} was reformulating~(\ref{HFHM}) at the continuous level introducing a penalization term to enforce the unit sphere constraint, which also requires modifying the expression of the
energy law~(\ref{energy}) including the potential of the penalization term itself. The penalty method was probably the first strategy for approximating~(\ref{HFHM}), and the most common penalty function is the Ginzburg--Landau function. The key idea is that an energy law can be obtained without explicitly using the unit sphere constraint. Yet, a significant drawback of this approach is that choosing a ``good'' value for the penalty parameter is far from trivial.

One fourth possibility was based on introducing in~(\ref{HFHM}) a Lagrange multiplier associated with the unit sphere constraint, hence obtaining a saddle point formulation.  To the best of our knowledge, the only numerical scheme using such an approach can be found in~\cite{Bartels_Lubich_Prohl_2009}, where the multiplier was chosen so that the unit sphere restriction was enforced at the nodal points, taking advantage of a closed nodal numerical integration rule. Regarding the time integration, a second-order algorithm based on a Crank--Nicolson
method was used to approximate the primary variable while the Lagrange multiplier was implicitly computed in terms of the primary variable itself. An unconditional energy law was obtained and convergence toward
weak solutions established. No inf-sup condition was proved in~\cite{Bartels_Lubich_Prohl_2009}, because at the time the finite element spaces and the estimates for the Lagrange multiplier were not well understood. In fact, the study of the inf-sup condition for the Lagrange multiplier in the saddle point formulation of~(\ref{HFHM}) is one of the main contributions of the present paper.

In addition to the above references, the interested reader is referred to~\cite{Kruzik-Prohl, Cimrak_2008} for two numerical surveys concerning specific topics for the Landau--Lifshitz--Gilbert equations. 

The first three methods mentioned above share one main
drawback, namely the fact that they are can not be easily modified when
a coupling term comes into play. For instance, the Ericksen--Leslie
equations consist of the Navier--Stokes equations with an additional
viscous stress tensor and a convective harmonic map heat flow equation.
In~\cite{Becker-Feng-Prohl_2008}, a numerical scheme is proposed for
solving them following the ideas in~\cite{Bartels_Prohl_2006}
and~\cite{Bartels_Prohl_2007}. However, the presence of the
convective term in the harmonic map heat flow equation prevents
fulfilling the discrete unit sphere condition, despite the possibility
to obtain \emph{a priori} energy estimates.
Instead, in~\cite{Badia_Guillen_Gutierrez_JCP}, a saddle point formulation
is presented for the Ericksen--Leslie equations enjoying a discrete
energy law and allowing a nodal enforcement of the unit sphere
constraint. Yet, the inf-sup condition was not well understood at the
time of writing~\cite{Badia_Guillen_Gutierrez_JCP}. In this work we present some ideas which lead to an inf-sup condition for the associated Lagrange multiplier of the scheme described in \cite{Badia_Guillen_Gutierrez_JCP}; thereby, the numerical analysis may be concluded.  

The goal of the present paper is to provide a saddle point framework for
approximating~(\ref{HFHM}) in which, using appropriate numerical
tools, an inf-sup stable finite element method can be constructed.
In this regard, it should be stressed that a proper choice of the finite
element spaces for the saddle point problem, namely one which results in
favorable estimates for the Lagrange multiplier, is extremely important
in order to ensure stability and avoid the so-called \emph{locking} of
the numerical solution, i.e.\ an unphysical stiffness of the computed
$\u$ field~\cite{Brezzi_Fortin_1991}.
To deal with the contradictory requirements mentioned above concerning the
regularity of the numerical solution on the one hand and, on the other hand, the
fulfillment of the unit sphere constraint, we propose to use first
order, conforming finite elements and to enforce the unit sphere
constraint at the finite element nodes. We show that, when combined with
a suitable closed quadrature rule, this ansatz results in a discrete
version of the energy low~(\ref{energy}). Hence, summarizing, we are
interested in a numerical algorithm which uses low-order finite
elements, preserves the unit length at the nodal points and satisfies a
discrete energy law and a discrete inf-sup condition for the discrete
Lagrange multiplier.

Moreover, we discuss two time integrators for our finite element saddle
point formulation. Indeed, it seems that there are very few time
integrators available which preserve a discrete energy law. In
particular, we will present one first-order time integrator based on a
semi-implicit Euler method and one second-order time integrator based on
the Crank--Nicolson method.  

The rest of the paper is organized as follows. In section \ref{sect:saddle-point-problem}, some notation is introduced, then we present the saddle-point formulation for (\ref{HFHM}) and prove an energy law for such formulation. Moreover, some inf-sup conditions for the Lagrange multiplier are established. In section 3 we set out our assumptions concerning the finite element spaces used to approximate the saddle-point formulation, and conclude with some results required in proving an inf-sup condition at the discrete level. In section~4 we present our numerical scheme discretized in space with the time being continuous. In section 5 we end up with some time realizations of the semi-discretized scheme that preserve the desired properties. Section~\ref{sect:implementation-details} deals with some specific implementation aspects of the fully discretized scheme. Finally, section~\ref{sect:numerical-results} is devoted to various computational experiments.

\section{Statement of the saddle-point problem}
\label{sect:saddle-point-problem}

\subsection{Notation} 
We will  assume the following notation throughout this paper. Let
$\mathcal{O}\subset \R^M$, with $M=2$ or $3$, be a Lebesgue-measurable
domain and let $1\le p\le\infty$. We denote by $L^p(\mathcal{O})$ the space
of all Lesbegue-measurable real-valued functions, $f:\mathcal{O}\to \R $,
being $p$th-summable in $\mathcal{O}$ for $p<\infty$  or essentially bounded
for $p=\infty$, and by $\|f\|_{L^p(\mathcal{O})}$ its norm. When $p=2$, the
$L^2(\mathcal{O})$ space is a Hilbert space whose inner product  is
denoted by $(\cdot,\cdot)$.

Let $\alpha = (\alpha_1, \alpha_2, . . . , \alpha_d)\in \mathds{N}^M$ be a
multi-index with $|\alpha|=\alpha_1+\alpha_2+...+\alpha_M$, and let
$\partial^\alpha$ be the differential operator such that
$$\partial^\alpha=
\Big(\frac{\partial}{\partial{x_1}}\Big)^{\alpha_1}...\Big(\frac{\partial}{\partial{x_d}}\Big)^{\alpha_d}.$$

For $m\ge 0$ and $1 \le p\le \infty $, we define $W^{m,p}(\mathcal{O})$ to be the Sobolev space of all functions whose $m$ derivatives are in $L^p(\mathcal{O})$, with the norm
\begin{align*}
\|f\|_{W^{m,p}(\mathcal{O})}&=\left(\sum_{|\alpha|\le m} \|\partial^\alpha f\|^p_{L^p(\mathcal{O})}\right)^{1/p} \quad  &&\hbox{for} \ 1 \leq p < \infty, \\
\|f\|_{W^{m,p}(\mathcal{O})}&=\max_{|\alpha|\le m} \|\partial^\alpha f\|_{L^\infty(\Omega)}, \quad  && \hbox{for} \  p = \infty,
\end{align*}
where $\partial^\alpha$ is understood in the distributional sense.
In the particular case of $p=2$, $W^{m,p}(\mathcal{O})=H^m(\mathcal{O})$. We also consider 
${\mathcal C}^0(\bar {\mathcal{O}})$ to be the space of continuous
functions on $\bar {\mathcal{O}}$. 

For any space $X$, we shall denote the vector space  $X^d$ by its
bold letter $\boldsymbol{X}$. For example, $(L^2(\mathcal{O}))^d$ is denoted by
$\boldsymbol{L}^2(\mathcal{O})$, $(H^m(\mathcal{O}))^d$ by $\H^m(\mathcal{O})$, etc.
Consistently, in order to distinguish  scalar-valued fields from vector-valued ones, we denote them by roman letters and bold-face letters, respectively. To shorten the notation, the norms $\|\cdot\|_{L^2(\Omega)}$ and $\|\cdot\|_{\boldsymbol{L}^2(\Omega)}$ are abbreviated $\|\cdot\|$; moreover, the dual space of $X$ is denoted by $X'$, with $\langle\cdot, \cdot\rangle$ indicating its dual pairing.

\subsection{Saddle-point formulation}
The saddle-point formulation for  (\ref{HFHM}) reads as follows: Find $\u: \Omega\times\R^+ \to \S^{N-1}$ and $q : \Omega\times\R^+ \to \R$ satisfying 
\begin{equation}\label{LM}
\left\{
\begin{array}{rccl}
        \partial_t \u
        -\gamma\Delta\u + \gamma q\u+ \alpha \u\times\partial_t\u  & = & \boldsymbol{0} &\mbox{ in }  \Omega\times \R^+,
\\
        |\u|^2&=&1 &\mbox{ in }  \partial\Omega\times \R^+.
\\
\partial_\n\u&=&\boldsymbol{0}&\mbox{ on } \partial\Omega\times \R^+,
\\
\u(0)&=&\u_0&\mbox{ in } \Omega. 
\end{array}
\right.
\end{equation}

The energy estimate associated with problem (\ref{LM})
was derived in \cite{Badia_Guillen_Gutierrez_JCP}. If we multiply $(\ref{LM})_1$ by $\partial_t \u$, and
integrate over $\Omega$, we  have
$$
\|\partial_t\u\|^2+\frac{1}{2}\frac{\rm d}{ {\rm d}t}\|\nabla\u\|^2
+\int_{\Omega}\partial_t\u (\x)\cdot
q (\x)\u (\x)\, {\rm d}\x  =0.
$$
To control the third term on the left hand side of the above equation, we take the time
derivative of $|\u|^2=1$. Thus, it follows that
$\partial_t\u\cdot\u=0$, i.e.\ $\partial_t\u$ and $\u$ are
orthogonal. Therefore,
\begin{equation}\label{Lagrange-Energy}
\|\partial_t\u\|^2+\frac{1}{2}\frac{\rm d}{ {\rm d}t}\|\nabla\u\|^2=0.
\end{equation}

The method under consideration is based on a variational formulation for (\ref{LM}) with $\u$ and $q$ as primary variables, where the unit sphere constraint is satisfied only at the nodes. This requirement is enough to prove a discrete version of an inf-sup condition.   
\subsection{ Inf-sup conditions}
The natural inf-sup condition for problem (\ref{LM}) is
\begin{equation}\label{Linfty-inf-sup}
\|q\|_{L^\infty(\Omega)'}\le  \sup_{\bar\u\in
\boldsymbol{L}^\infty(\Omega)\setminus
\{\boldsymbol{0}\}}\frac{\left<q,\u\cdot\bar\u\right>}{\|\bar\u\|_{\boldsymbol{L}^\infty(\Omega)}}\quad
\forall\, q\in L^{\infty}(\Omega)',
\end{equation}
since $q=-|\nabla\u|^2\in L^\infty(0,T; L^1(\Omega))$ and $L^1(\Omega)\subset L^\infty(\Omega)'$. To prove such an inf-sup condition 
(\ref{Linfty-inf-sup}) one needs to make the assumption that $|\u|=1$ holds a.e.\ in $\Omega$. Under this assumption, let us first see that the mapping $\u\cdot : \boldsymbol{L}^\infty(\Omega)\to L^\infty(\Omega)$ is surjective. Indeed, let $e\in L^\infty(\Omega)$, then choose $\bar\u=\u\,e$. Clearly, $e=\u\cdot\bar\u\in L^\infty(\Omega)$. Next, observe that $\|\bar\u\|_{L^\infty(\Omega)}\le \|e\|_{L^\infty(\Omega)}$.
Thus, we have
$$
\|q\|_{L^\infty(\Omega)'}= \sup_{e\in L^\infty(\Omega)\setminus
\{0\}}\frac{\left<q,e\right>}{\|e\|_{L^\infty(\Omega)}}\le
\sup_{\bar\u\in \boldsymbol{L}^\infty(\Omega)\setminus
\{\boldsymbol{0}\}}\frac{\left<q,\u\cdot\bar\u\right>}{\|\bar\u\|_{\boldsymbol{L}^\infty(\Omega)}}
$$
for all $q\in L^\infty(\Omega)'$. This inf-sup condition however is not applicable because, due to the presence of $-\Delta\u$ in~$(\ref{LM})_1$ which can not be bounded in $\boldsymbol{L}^\infty(\Omega)'$. Therefore, we need to weaken the norm for the Lagrange multiplier $q$. Now, the mapping $\u\,\cdot : \boldsymbol{L}^\infty(\Omega)\cap \H^1(\Omega)\to L^\infty(\Omega)\cap H^1(\Omega)$ is surjective by assuming  $\u\in \boldsymbol{L}^\infty(\Omega)\cap \H^1(\Omega)$  such that $|\u|=1$ a.e.\ in $\Omega$.  Moreover, there exists a positive constant $C=C(\u)$ such that $\|\nabla \bar\u\|\le C \|\nabla e\|$ for $e\in L^\infty(\Omega)\cap H^1(\Omega)$.  Thus, if  $q\in
(\H^1(\Omega)\cap \boldsymbol{L}^\infty(\Omega))'$, then one can prove 
\begin{equation}\label{H1Linf-inf-sup_bis} 
\|q\|_{(H^1(\Omega)\cap L^\infty(\Omega))'}\le C
\sup_{\bar\u\in \H^1(\Omega)\cap \boldsymbol{L}^\infty(\Omega) \setminus
\{\boldsymbol{0}\}}\frac{\left<q,\u\cdot\bar\u\right>}{\|\nabla\bar\u\|+\|\bar\u\|_{\boldsymbol{L}^\infty(\Omega)}}.
\end{equation}

From a numerical point of view, one must be aware that the difficulty lies in establishing the counterpart of such an inf-sup condition  at the discrete level.

\section{Spatial discretization}

\subsection{Finite element spaces}
Herein we introduce the hypotheses that will be required along this work.
\begin{enumerate}
\item [(H1)] Let $\Omega$ be a bounded domain of $\R^M$ 
 with a polygonal or polyhedral Lipschitz-continuous boundary.

\item[(H2)] Let $\{{\mathcal T}_{h}\}_{h>0}$  be a family of shape-regular, quasi-uniform triangulations of 
$\overline{\Omega}$ made up of triangles in two dimensions 
and tetrahedra in three dimensions, so that 
$\overline \Omega=\cup_{K\in {\mathcal T}_h}K$, where $h=\max_{K\in \mathcal{T}_h} h_K$, with $h_K$ being the diameter of $K$. Further,  let 
${\mathcal N}_h = \{\a_i\}_{i\in I}$ denote the set of all nodes of ${\mathcal T}_h$.

\item [(H3)] Conforming finite-element spaces associated with 
${\mathcal T}_h$ are assumed for approximating $H^1(\Omega)$. Let  $\mathcal{P}_1(K)$ be the set of linear polynomials on  $K$; the space of continuous, piecewise polynomial 
functions on ${\mathcal T}_h$  is then denoted as
$$
X_h = \left\{ v_h \in {C}^0(\overline\Omega) \;:\; 
v_h|_K \in \mathcal{P}_1(K), \  \forall K \in \mathcal{T}_h \right\}.
$$
For $v_h\in X_h$, we denote the nodal values by $v_h(\a) = v_{\a}$.
Also, we identify the Lagrangian basis functions of $X_h$ through the
node where they do not vanish, using the notation $\varphi_{\a}$, so
that $v_h = \sum_{\a\in\mathcal{N}}v_{\a}\varphi_{\a}$ and, for vector
valued functions, $\v_h = \sum_{i=1}^M \e_i \sum_{\a\in\mathcal{N}}
v^i_{\a}\varphi_{\a} = \sum_{\a\in\mathcal{N}} \v_{\a}\varphi_{\a}$,
with $\e_i$ being the unit vectors of the canonical basis in $\R^M$.
\item [(H4)]  We assume that $\u_0\in \H^1(\Omega)$ with $|\u_0|=1$ a.e.\ in $\Omega$. Then we consider $\u_{0h}\in \U_h$ such that $ |\u_{0h}(\a)|=1$ for all $\a\in\mathcal{N}_h$ and $\|\nabla\u_{0h}\|\le \|\nabla\u_0\|$.
\end{enumerate}

We choose the following continuous finite-element spaces 
$$
\U_h=\X_h\quad\hbox{and}\quad Q_h=X_h
$$
to approximate the vector field and the Lagrange multiplier, respectively.

In proving a discrete inf-sup condition we will need to set out some commuter properties for the nodal projection operator into $\U_h$. Although these properties were already obtained in \cite{Johnson_Szepessy_Hansbo_1990}, the proof of such properties will be helpful to see that the above assumptions are enough for our purpose.  

To start with, some inverse inequalities are provided in the following proposition (see e.g.~\cite[Lm 4.5.3]{Brenner_Scott_2008} or \cite[Lm 1.138]{Ern_Guermond_2004}).          
\begin{proposition} Under hypotheses $\rm(H1)$--$\rm(H3)$, it follows that, for all  $x_h\in \mathcal{P}_1(K)$, 
\begin{equation}\label{inv_H1toL2}
\|\nabla x_h \|_{\boldsymbol{L}^2(K)}\le C_{\rm inv} h_K^{-1} \|x_h\|_{L^2(K)},
\end{equation}
and
\begin{equation}\label{inv_W1inftoLinf}
\|\nabla x_h \|_{\boldsymbol{L}^\infty(K)}\le C_{\rm inv} h_K^{-1} \|x_h\|_{L^\infty(K)},
\end{equation}
where $C_{\rm inv}>0$ is a constant independent of $h$ and $K$.
\end{proposition}

For  each $K\in\mathcal{T}_h$, let $i_K$ be the local nodal interpolation operator defined from $C^0(K)$ into $\mathcal{P}_1(K)$, and let $i_{X_h}$ be the associated global nodal interpolation operator from $C^0(\bar \Omega)$ into $X_h$, i.e.\ $i_{K}:=i_{X_h}|_K$, for all $K\in \mathcal{T}_h$. Moreover, let $\pi^0_{K}$ be the $L^2(K)$ orthogonal projection operator from $L^1(\Omega)$ onto $\mathcal{P}_0$, where $\mathcal{P}_0$ is the set of constant polynomials on $K$. Next we give some local error estimates for these two local interpolants. See e.g. \cite[Th 4.4.4]{Brenner_Scott_2008} or \cite[Th 1.103]{Ern_Guermond_2004}.
\begin{proposition} Suppose that  hypotheses $\rm (H1)$--$\rm (H3)$ hold. Then the local nodal interpolation operator $i_{K}$ satisfies  
\begin{equation}\label{interp_error_nodal_L2_and_H2_for_Xh}
\|\varphi-i_{K} \varphi \|_{L^2(K)} \leq C_{\rm app}\, h_K^2\|\nabla^2 \varphi\|_{\boldsymbol{L}^2(K)}\quad\mbox{ for all }\quad \varphi\in H^2(K)
\end{equation}
and
\begin{equation}\label{interp_error_nodal_Linf_and_W1inf_for_Xh}
\|\varphi-i_{K} \varphi \|_{L^\infty(K)} \leq C_{\rm app}\, h_K\|\nabla \varphi\|_{\boldsymbol{L}^\infty(K)}\quad\mbox{ for all }\quad\varphi\in W^{1,\infty}(K),
\end{equation}
where $C_{app}>0$ is a constant independent of $K$ and $h_K$. 
\end{proposition}
\begin{proposition} Under  hypotheses $\rm (H1)$--$\rm (H3)$, it follows that,  for all $ x_h, y_h\in X_h$, 
\begin{equation}\label{interp_error_nodal_L2_and_H1_for_Xh}
\|x_h y_h-i_{K} (x_h y_h)\|_{L^2(K)} \leq C_{\rm app}\, h_K\|\nabla (x_h y_h)\|_{\boldsymbol{L}^2(K)},
\end{equation}
where $C_{app}>0$ is a constant independent of $K$ and $h_K$.
\end{proposition}
\begin{proof} 
Estimate \eqref{interp_error_nodal_L2_and_H1_for_Xh} follows readily from \eqref{interp_error_nodal_L2_and_H2_for_Xh} and \eqref{inv_H1toL2}, upon observing that
the components of $\nabla(x_h y_h)$ belong to $\mathcal{P}_1(K)$.      
\end{proof}

\begin{proposition} Suppose that hypotheses $\rm (H1)$--$\rm (H3)$ hold. Then $\pi^0_{K}$ satisfies  
\begin{equation}\label{interp_error_pi_L2_and_H1_for_Xh0}
\| \varphi-\pi^0_{K} \varphi\|_{L^2(K)} \leq C_{\rm app}\, h_K\|\nabla \varphi\|_{\boldsymbol{L}^2(K)}\quad
\quad\mbox{ for all }\quad\varphi\in H^1(K),
\end{equation}
and
\begin{equation}\label{interp_error_pi_Linf_and_W1inf_for_Xh0}
\| \varphi-\pi^0_{K} \varphi\|_{L^\infty(K)} \leq C_{\rm app}\, h_K\|\nabla \varphi\|_{\boldsymbol{L}^\infty(K)}\quad
\quad\mbox{ for all }\quad\varphi\in H^1(K),
\end{equation}
where $C_{\rm app}>0$ is constant independent of $h_K$. 
\end{proposition}
Let $\pi_{Q_h}$ denote the $L^2(\Omega)$-orthogonal projection operator from $L^2(\Omega)$ into $Q_h$.  The following proposition deals with the stability of $\pi_{Q_h}$. See \cite{Douglas_Dupont_Wahlbin_1974} and \cite[Lm 1.131]{Ern_Guermond_2004}.
\begin{proposition} Suppose that assumptions $\rm(H1)$--$\rm(H3)$ are satisfied. Then there exists a positive  constant  $C_{\rm sta}$, independent of $h$, such that
\begin{equation}\label{stab-pi-Linf-Q_h}
\|\pi_{Q_h} \varphi\|_{L^\infty (\Omega)} \le C_{\rm sta} \| \varphi\|_{L^\infty (\Omega)}\quad\mbox{ for all }\quad\varphi\in L^\infty(\Omega),
\end{equation}
and
\begin{equation}\label{stab-pi-H1-Qh}
\|\nabla\pi_{Q_h} \varphi\| \le C_{\rm sta} \|\nabla \varphi\|\quad\mbox{ for all }\quad\varphi\in H^1(\Omega).
\end{equation}

\end{proposition}
We will prove discrete commuter properties for $i_{X_h}$, following very closely the arguments of \cite{Johnson_Szepessy_Hansbo_1990}.
\begin{proposition} Assume that hypotheses $\rm(H1)$--$\rm(H3)$ hold and let  $x_h, y_h\in X_h$. Then there exists a constant $C_{\rm com}>0$, independent of $h$ and $K$, such that 
\begin{equation}\label{commuter-Linfty}
\|x_h y_h- i_{K}(x_h y_h)\|_{L^\infty(K)}\le C_{\rm com}\, h_K \|x_h\|_{L^\infty(K)} \|\nabla y_h\|_{\boldsymbol{L}^\infty(K)},
\end{equation}
and 
\begin{equation}\label{commuter-H1}
\|\nabla (x_h y_h-i_{K}(x_h y_h))\|_{\boldsymbol{L}^2(K)}\le C_{\rm com}  \|x_h\|_{L^\infty(K)} \|\nabla y_h\|_{\boldsymbol{L}^2(K)}
\end{equation}
hold for all $K\in\mathcal{T}_h$.  
\end{proposition}
\begin{proof}  Using the triangle inequality, we bound           
\begin{equation}\label{pro5-lab1}
\begin{array}{rcl}
\|i_{K}( x_h y_h)- x_h y_h\|_{L^\infty(K)}&\le& \|i_{K}( x_h y_h)- i_{K}(x_h\pi^0_{K}(y_h))\|_{L^\infty(K)}
\\
&&+\|i_{K}(x_h\pi_{K}^0(y_h)- x_h\pi_{K}^0(y_h))\|_{L^\infty(K)}
\\
&&+\|x_h \pi_{K}^0(y_h)- x_hy_h\|_{L^\infty(K)}.
\end{array}
\end{equation}
The first term on the right-hand side of \eqref{pro5-lab1} can be estimated as follows:
\begin{equation}\label{pro5-lab2}
\begin{array}{rcl}
\|i_{K}( x_h y_h)- i_{K}(x_h\pi^0_{K}(y_h))\|_{L^\infty(K)}&=&\|i_{K}(x_h y_h-  x_h\pi^0_{K}(y_h))\|_{L^\infty(K)}
\\
&\le& \|i_{K}(x_h y_h- x_h\pi_{K}^0(y_h))-  (x_h y_h-x_h\pi_{K}^0(y_h))\|_{L^\infty(K)}
\\
&&+\|x_h\pi_{K}^0(y_h)-x_h y_h\|_{L^\infty(K)}.
\end{array}
\end{equation}
Thus, by (\ref{interp_error_nodal_Linf_and_W1inf_for_Xh}), (\ref{interp_error_pi_Linf_and_W1inf_for_Xh0}) and (\ref{inv_W1inftoLinf}), we have
\begin{equation}
\begin{array}{rcl}
\|i_{K}(x_h\pi_{K}^0(y_h)-x_h y_h)-  (x_h\pi_{K}^0(y_h)-x_h y_h)\|_{L^\infty(K)} &\le& C h_K \|\nabla (x_h\pi_{K}^0(y_h)-x_h y_h)\|_{\boldsymbol{L}^\infty(K)}
\\
&\le& C h_K \|x_h\|_{L^\infty (\Omega)} \|\nabla y_h\|_{\boldsymbol{L}^\infty(K)} 
\end{array}
\end{equation}
and 
$$
\|x_h\pi_h^0(y_h)-x_h y_h\|_{L^\infty(K)}	\le C h_K \|x_h\|_{L^\infty(K)} \|\nabla y_h\|_{\boldsymbol{L}^\infty(K)}.
$$
Therefore, 
$$
\|i_{K}( x_h y_h)- i_{K}(x_h\pi^0_{K}(y_h))\|_{L^\infty(K)}\le C h_K \|x_h\|_{L^\infty(K)} \|\nabla y_h\|_{\boldsymbol{L}^\infty(K)}.
$$

Now observe that the second term in the right-hand side of \eqref{pro5-lab1} is zero since $i_{K}(x_h\pi_{K}^0(y_h))=x_h\pi_{K}^0(y_h)$. And the third term can be easily estimated as before. In view of the above computations, one can conclude that (\ref{commuter-Linfty}) holds.   

Similarly, we have 
\begin{equation}\label{pro5-lab3}
\begin{array}{rcl}
\|\nabla(i_{K}( x_h y_h)- x_h y_hy_h)\|_{\boldsymbol{L}^2(K)}&\le& \|\nabla(i_{K}( x_h y_h)- i_{K}(x_h\pi_{K}^0(y_h)))\|_{\boldsymbol{L}^2(K)}
\\
&&+\|\nabla(i_{K}(x_h\pi_{K}^0(y_h))- x_h\pi_{K}^0(y_h))\|_{\boldsymbol{L}^2(K)}
\\
&&+\|\nabla(x_h \pi_{K}^0(y_h)- x_hy_h)\|_{\boldsymbol{L}^2(K)}.
\end{array}
\end{equation} 
From (\ref{inv_H1toL2}), (\ref{interp_error_nodal_L2_and_H1_for_Xh}), (\ref{inv_W1inftoLinf}) and (\ref{interp_error_pi_L2_and_H1_for_Xh0}), we obtain
$$
\begin{array}{rcl}
\|\nabla(i_{K}( x_h y_h)- i_{K}(x_h\pi_{K}^0(y_h)))\|_{\boldsymbol{L}^2(K)}&\le& C h_K^{-1} \|i_{K}( x_h y_h)- i_{K}(x_h\pi_{K}^0(y_h))\|_{\boldsymbol{L}^2(K)}\
\\
&\le& C \|x_h\|_{L^\infty(K)} \|\nabla y_h\|_{\boldsymbol{L}^2(K)},
\end{array}
$$
where we have argued as in estimating (\ref{pro5-lab2}), but for the $L^2(K)$-norm. The second term on the right-hand side of (\ref{pro5-lab3}) is zero again. To control the last term of \eqref{pro5-lab3}, we have, by (\ref{inv_W1inftoLinf}) and  (\ref{interp_error_pi_L2_and_H1_for_Xh0}), that 
$$
\begin{array}{rcl}
\|\nabla(x_h \pi_{K}^0(y_h)- x_hy_h)\|_{\boldsymbol{L}^2(K)}&=&\|\nabla(x_h (\pi_{K}^0(y_h)-y_h))\|_{\boldsymbol{L}^2(K)}
\\
&\le&\|\nabla x_h\|_{\boldsymbol{L}^\infty(K)} \|\pi_{K}^0(y_h)-y_h\|_{L^2(K)}
\\
&&+\|x_h\|_{L^\infty(K)} \|\nabla (\pi_{K}^0(y_h)-y_h)\|_{\boldsymbol{L}^2(K)} 
\\
&\le&C \|x_h\|_{L^\infty(K)} \|\nabla y_h\|_{\boldsymbol{L}^2(K)}. 
\end{array}
$$   
\end{proof}
\begin{remark}\label{Rm:global} The global version of the above propositions holds due to the assumed quasi-uniformity for the mesh $\mathcal{T}_h$.
\end{remark}

Let us define 
$$
(\u_h,\bar\u_h)_h=\int_\Omega i_{Q_h}(\u_h\cdot\bar\u_h)=\sum_{\a\in\mathcal{N}_h} \u_h(\a)\cdot\bar\u_h(\a) \int_\Omega \varphi_\a 
$$
for all $\u_h,\bar\u_h\in \U_h$, with the induced norm $\|\u_h\|_h = \sqrt{(\u_h,\u_h)_h}$. 


\section{Numerical scheme}
In this section we will propose our numerical method and will prove a discrete energy law and a discrete inf-sup condition for the Lagrange multiplier similar to (\ref{Lagrange-Energy}) and (\ref{H1Linf-inf-sup_bis}), respectively, at the continuous level. The main results are given in Lemma \ref{main1} and Corollary \ref{main2} which is a consequence of Lemma \ref{main3}.

The numerical approximation under consideration is based on a conforming  finite element method for the variational formulation of (\ref{LM}).  Then we want to find $(\u_h, q_h)\in C^\infty([0,+\infty);\U_h)\times C^\infty([0,+\infty); Q_h)$ such that, for all $(\bar\u_h, \bar q_h)\in \U_h\times Q_h$, 
\begin{equation}\label{FEM}
\left\{
\begin{array}{rcl}
(\partial_t\u_h, \bar\u_h)_h+ \gamma(\nabla\u_h, \nabla\bar\u_h)+\gamma(q_h, i_{Q_h}(\u_h\cdot\bar\u_h))+\alpha(\u_h\times\partial_t\u_h, \bar\u_h)_h&=&0,
\\
(i_{Q_h}(\u_h\cdot\u_h), \bar q_h )&=&(1, \bar q_h),
\end{array}
\right.
\end{equation}
with 
$$
\u_h(0)=\u_{0h}\quad \mbox{ in }\quad \Omega,
$$
where $\u_{0h}\in \U_h$ is defined as in  $({\rm H4})$.

\begin{remark}
How to obtain an approximate initial condition $\u_{0h}\in\U_h$ such that $\|\nabla\u_{0h}\|\le C \|\nabla \u_0\|$ and $|\u_{0h}(\a)|=1$ for all $\a\in\mathcal{N}_h$  is rarely explicitly mentioned in numerical papers.  It seems that this condition is overlooked. Nevertheless, it is very important as it can be checked in the proof of Lemma \ref{lm:local_existence} below. For instance, these conditions can be achieved by appling the nodal interpolation operator $i_{\U_h}$ to $\u_0\in \boldsymbol{C}^0(\bar \Omega)$ and by assuming $(\rm H5)$ in Section \ref{sec5}. 

Avoiding the $C^0(\bar \Omega)$-regularity to obtain $({\rm H4})$ is an interesting open problem in the numerical framework of the Landau--Lifshitz--Gilbert and harmonic map heat flow equation.
\end{remark}

%

Next we consider the local-in-time well-posedness of  (\ref{FEM}).
\begin{lemma}\label{lm:local_existence} There exists $T_h>0$, depending possibly on $h$, such that there is a unique solution to problem (\ref{FEM}) on $[0,T_h)$. 
\end{lemma}
\begin{proof}
The proof includes two steps: first we show that~(\ref{FEM}) is
equivalent to a system of ordinary differential equation, then we show
that such a system has a unique solution.

Let us assume that $\u_h,q_h$ is a solution of~(\ref{FEM}). Pick $\bar\a
\in \mathcal{N}_h$ and take $\bar\u_h=\u_h(\bar\a,
t)\varphi_{\bar\a}=\u_{\bar\a}\varphi_{\bar\a}$ in~(\ref{FEM})$_1$ to get
$$
(\varphi_{\bar\a},\varphi_{\bar\a}) \frac{\rm d}{ {\rm d}t}|\u_{\bar a}|^2 +
\gamma \sum_{\a\in\mathcal{N}_h} \u_{\a}\cdot\u_{\bar\a}
(\nabla\varphi_{\a},\nabla\varphi_{\bar\a})
+\gamma \sum_{\a\in\mathcal{N}_h}q_{\a}(\varphi_{\a},|\u_{\bar\a}|^2
\varphi_{\bar\a})=0.
$$
Using now~(\ref{FEM})$_2$, we conclude $|\u_{\bar\a}|=1$, so that the
first term vanishes and we obtain
\begin{equation}\label{lm4.1-lab1}
\sum_{\a\in\mathcal{N}_h} \u_{\a}\cdot\u_{\bar\a}
(\nabla\varphi_{\a},\nabla\varphi_{\bar\a})
+ \sum_{\a\in\mathcal{N}_h}q_{\a}(\varphi_{\a},\varphi_{\bar\a})=0.
\end{equation}
The coefficients multiplying $q_{\a}$ in~(\ref{lm4.1-lab1}) define a
nonsingular matrix (the mass matrix of $Q_h$), so that this equation can
be used to compute $q_{\a}$ uniquely in terms of $\u_{\a}$. Next take
$\bar\u_h = \e_i\varphi_{\bar\a}$ in~(\ref{FEM})$_1$ to obtain
$$
\mu_{\bar\a} (u_{\bar\a}^i)' 
+ \gamma\sum_{\a\in \mathcal{N}_h} u_{\a}^i 
(\nabla \varphi_{\a},\nabla\varphi_{\bar\a})
+ \gamma\sum_{\a\in\mathcal{N}_h} q_{\a} (\varphi_{\a},
u^i_{\bar \a}\varphi_{\bar \a})
+ \alpha\mu_{\bar a} \u_{\bar\a} \times (\u_{\bar\a})' \cdot \e_i = 0,
$$
where we have introduced $\mu_{\bar\a}=\int_{\Omega}\varphi_{\bar\a}$.
Observe that
$$
\u_{\bar\a} \times (\u_{\bar\a})' \cdot \e_i = \mathcal{U}_{\times,\bar\a}
(\u_{\bar\a})' \cdot \e_i =
\sum_{j=1}^M (\mathcal{U}_{\times,\bar\a})_{ij} (u^j_{\bar\a})',
$$
where $\mathcal{U}_{\times,\bar\a}$ is the skew-symmetric matrix
representing the vector product, i.e.
$$
\mathcal{U}_{\times,\bar\a} = \left(
\begin{array}{ccc}
0 & -u^3_{\bar\a} & u^2_{\bar\a}
\\
u^3_{\bar\a}& 0 & -u^1_{\bar\a}
\\
-u^2_{\bar\a} & u^1_{\bar\a} & 0
\end{array}
\right).
$$
Using now~(\ref{lm4.1-lab1}) one arrives at
\begin{equation} \label{lm4.1-lab3}
\mu_{\bar\a}\left( \mathcal{I} + \alpha \mathcal{U}_{\times,\bar\a}
\right) (\u_{\bar\a})' 
+ \gamma\sum_{\a\in \mathcal{N}_h}
(\nabla \varphi_{\a},\nabla\varphi_{\bar\a}) \u_{\a}
- \gamma\sum_{\a\in\mathcal{N}_h}  \u_{\a}\cdot\u_{\bar\a} (\nabla
\varphi_{\a}, \nabla \varphi_{\bar \a}) \u_{\bar \a} = \boldsymbol{0},
\end{equation}
with $\mathcal{I}$ being the identity matrix. One can easily verify that
$\left( \mathcal{I} + \alpha \mathcal{U}_{\times,\bar\a}\right)$ is a
nonsingular matrix with determinant $1+\alpha^2|\u_{\bar\a}|^2$, so
that~(\ref{lm4.1-lab3}) defines a system of
ordinary differential equation which is locally Lipschitz continuous in
the coefficients $u^i_{\bar\a}$, for all $\bar\a\in\mathcal{N}_h$. 

Conversely, let us assume that $\u_h,q_h$ are defined by the nodal
coefficient solving~(\ref{lm4.1-lab3}) and~(\ref{lm4.1-lab1}).
Substituting~(\ref{lm4.1-lab1}) in~(\ref{lm4.1-lab3}) immediately
provides~(\ref{FEM})$_1$. To see that in fact also~(\ref{FEM})$_2$ is
verified, proceed as follows. For each $\bar\a\in\mathcal{N}_h$,
multiply~(\ref{lm4.1-lab3}) by $\u_{\bar\a}$ and observe that
$\u_{\bar\a}^T \mathcal{U}_{\times,\bar\a}$ vanishes, so that
$$
\frac{\mu_{\bar\a}}{2}\frac{\rm d}{ {\rm d}t}\left( 1 - |\u_{\bar\a}|^2
\right)
+ \gamma (1-  |\u_{\bar\a}|^2) 
\sum_{\a\in\mathcal{N}_h}  \u_{\a}\cdot\u_{\bar\a} (\nabla
\varphi_{\a}, \nabla \varphi_{\bar \a}) = 0.
$$
Define now
$$
g_{\bar\a}(t) = \frac{2}{\gamma \mu_{\bar\a}} \sum_{\a\in\mathcal{N}_h}
\u_{\a}(t) \cdot \u_{\bar\a}(t)
(\nabla\varphi_{\a},\nabla\varphi_{\bar\a})
$$ 
so that
$$
\frac{\rm d}{ {\rm d}t}\left( 1 - |\u_{\bar\a}|^2 \right)
+ g_{\bar\a} (1-  |\u_{\bar\a}|^2) = 0,
$$
with solution
$$
(1-|\u_{\bar\a}(t)|^2) = e^{\displaystyle \int_0^t g_{\bar\a}(s)
{\rm d}s} (1- |\u_{\bar\a}(0)|^2).
$$
Hence, assuming that $|\u_{\bar\a}(0)| = 1$, (\ref{FEM})$_2$ holds for
any time.

To complete the proof we need to show that~(\ref{lm4.1-lab3}) has a
unique solution on $[0, T_h)$, which follows from Picard's theorem. 
\end{proof}
\begin{remark} As a result of the \emph{a priori} estimates for $(\u_h, q_h)$ in the next section, the approximated solution $(\u_h, q_h)$ will exist globally in time on $\R^+$. 
\end{remark}
\begin{remark} Equation (\ref{lm4.1-lab3}) gives us a way to compute $\u_h$ without using the Lagrange multiplier $q_h$. This way $q_h$ is somehow an approximation of $-|\nabla\u_h|^2$. 
\end{remark}

In the following lemma, we prove a pointwise estimate and \emph{a priori} energy estimates for (\ref{FEM}).
\begin{lemma}\label{main1} Assume that assumptions $\rm(H1)$--$\rm(H4)$ hold. Then the discrete solution $\u_h$ of scheme (\ref{FEM}) satisfies 
\begin{equation}\label{bound_Linfty_for_uh}
|\u_h(\a)|=1 \quad \mbox{for all}\quad \a\in {\mathcal N}_h,
\end{equation}
and
\begin{equation}\label{bound_energy_for_uh}
\int_0^t\|\partial_t\u_h(s)\|^2_h {\rm d}s+\frac{\gamma}{2}\|\nabla\u_h(t)\|^2=\frac{\gamma}{2}\|\nabla\u_{0h}\|^2 \quad \mbox{for all}\quad t\in \R^+.
\end{equation}
\end{lemma}
\begin{proof}  The nodal equality  (\ref{bound_Linfty_for_uh}) follows easily from $(\ref{FEM})_2$ since $i_{Q_h}(\u_h\cdot\u_h)(\a)=1$ for all $\a\in\mathcal{N}_h$.

Selecting $\bar\u_h=\partial_t\u_h$ in $\eqref{FEM}_1$, we obtain
\begin{equation}\label{lm4.4-lab1}
\|\partial_t\u_h\|^2_h+\frac{\gamma}{2}\frac{{\rm d}}{{\rm d}t}\|\nabla\u_h\|^2+\gamma(q_h,i_{Q_h}(\u_h\cdot\partial_t\u_h))=0.
\end{equation}
Now differentiating $\eqref{FEM}_2$ with respect to $t$ and then setting $\bar q_h=q_h$  yields
$$
2(i_{Q_h}(\partial_t\u_h\cdot\u_h), q_h)=0.
$$
Using this in \eqref{lm4.4-lab1}, we have
\begin{equation}\label{lm4.4-lab2}
\|\partial_t\u_h\|^2_h+\frac{\gamma}{2}\frac{{\rm d}}{{\rm d}t}\|\nabla\u_h\|^2=0.
\end{equation}
Then we have that \eqref{bound_energy_for_uh} holds by integrating \eqref{lm4.4-lab2}. 
\end{proof}
\begin{remark} The satisfaction of the unit sphere constraint at the nodes along with the fact that $\u_h$ is a piecewise linear finite element solution implies a uniform pointwise estimate for $\u_h$, i.e., that $\|\u_h\|_{\boldsymbol{L}^\infty(\Omega)}\le 1$.    
\end{remark}
The next lemma deals with the discrete inf-sup condition for (\ref{FEM}).
\begin{lemma}\label{main3} Assume that assumptions $\rm(H1)$--$\rm(H4)$ hold.  Let $\u_h\in\U_h$ such that $|\u_h(\a)|=1$ for all $\a\in{\mathcal N}_h$. Then the following inf-sup condition holds:
\begin{equation}\label{inf-sup}
C \frac{1}{1+\|\nabla \u_h\|} \le 
\inf_{q_h\in Q_h}\sup_{\bar\u_h\in \U_h\setminus
\{\boldsymbol{0}\}}\frac{(q_h, i_{Q_h} (\u_h\cdot \bar\u_h))}{\|q_h\|_{(H^1(\Omega)\cap L^\infty(\Omega))'}(\|\nabla\bar\u_h\|+\|\bar\u_h\|_{\boldsymbol{L}^\infty(\Omega)})},
\end{equation}
where $C>0$ is a constant independent of $h$.
\end{lemma}
\begin{proof} Let $ q\in H^1(\Omega)\cap L^\infty(\Omega)$. Take $\bar\u_h=\boldsymbol{i}_{\U_h}(\u_h\pi_{Q_h}(q))$, where $ \boldsymbol{i}_{\U_h}$ is the nodal interpolation operator into $\U_h$ and $\pi_{Q_h} $ is the $L^2(\Omega)$ orthogonal projection operator onto $Q_h$, and observe that 
$$
\begin{array}{rcl}
(q_h, i_{Q_h}(\u_h\cdot\bar\u_h))&=&\displaystyle(q_h, \sum_{\a\in\mathcal{ N}_h} \pi_{Q_h}(q)|_{\a} \u_\a\cdot\u_\a \varphi_\a )=\sum_{\a\in\mathcal{ N}_h}(q_h, \pi_{Q_h}(q)|_{\a} \varphi_\a )
\\
&=&\displaystyle(q_h, i_{Q_h}(\pi_{Q_h}(q)))=(q_h, \pi_{Q_h}(q))=(q_h, q).
\end{array}
$$
Then we obtain
$$
\sup_{\bar\u_h\in \U_h\setminus
\{\boldsymbol{0}\}}\frac{(q_h, i_{Q_h} (\u_h\cdot \bar\u_h))}{\|\nabla\bar\u_h\|+\|\bar\u_h\|_{\boldsymbol{L}^\infty(\Omega)}}\ge \sup_{q\in H^1(\Omega)\cap L^\infty(\Omega)\setminus
\{0\}}\frac{(q_h, q)}{\|\nabla\boldsymbol{i}_{\U_h}(\u_h\pi_{Q_h}(q))\|+\|\boldsymbol{i}_{\U_h}(\u_h\pi_{Q_h}(q))\|_{\boldsymbol{L}^\infty(\Omega)}}.
$$
Moreover, we have, by (\ref{commuter-Linfty}) and (\ref{commuter-H1}),  that  
\begin{equation}\label{est-Linf}
\|\boldsymbol{i}_{\U_h}(\u_h\pi_{Q_h}(q))\|_{\boldsymbol{L}^\infty(\Omega)}\le C \|\u_h\|_{\boldsymbol{L}^\infty(\Omega)} \|q\|_{L^\infty(\Omega)}
\end{equation}
and 
\begin{equation}\label{est-H1}
\|\nabla\boldsymbol{i}_{\U_h}(\u_h\pi_{Q_h}(q))\|\le C ( \|\nabla\u_h\| \|q\|_{L^\infty(\Omega)} + \|\nabla q\| \|\u_h\|_{\boldsymbol{L}^\infty(\Omega)} )
\end{equation}
due to Remark \ref{Rm:global}. Observe also that we have utilized (\ref{stab-pi-Linf-Q_h}) and (\ref{stab-pi-H1-Qh}). Therefore, 
\begin{equation}\label{lm4.6-lab1}
\|\nabla\boldsymbol{i}_{\U_h}(\u_h\pi_{Q_h}(q))\|+\|\boldsymbol{i}_{\U_h}(\u_h\pi_{Q_h}(q))\|_{\boldsymbol{L}^\infty(\Omega)}\le C (1+\|\nabla \u_h\|) (\|q\|_{L^\infty(\Omega)}+\|\nabla q\|).
\end{equation}
As a result, we find
$$
\begin{array}{rcl}
\displaystyle
\sup_{\bar\u_h\in \U_h\setminus
\{\boldsymbol{0}\}}\frac{(q_h, i_{Q_h} (\u_h\cdot \bar\u_h))}{\|\nabla\bar\u_h\|+\|\bar\u_h\|_{\boldsymbol{L}^\infty(\Omega)}}&\ge&\displaystyle C \frac{1}{1+\|\nabla \u_h\|} \sup_{q\in H^1(\Omega)\cap L^\infty(\Omega)\setminus
\{0\}}\frac{(q_h, q)}{\|\nabla q\|+\|q\|_{L^\infty(\Omega)}}
\\
&\ge&\displaystyle C \frac{1}{1+\|\nabla \u_h\|} \|q_h\|_{(H^1(\Omega)\cap L^\infty(\Omega))'}.
\end{array}
$$
Then the proof follows by taking infimum over $Q_h$.
\end{proof}

\begin{corollary}\label{main2} Assume that assumptions $\rm(H1)$--$\rm(H4)$ hold. The discrete Lagrange multiplier $q_h$ of scheme (\ref{FEM}) satisfies
\begin{equation}\label{est:q_hinL2}
 \|q_h\|_{L^2(0,T;(H^1(\Omega)\cap L^\infty(\Omega))')} \le C (1+\|\nabla \u_0\|) \|\nabla \u_0\|.
\end{equation}

\end{corollary}
\begin{proof} From $(\ref{FEM})_1$, we have 
$$
\begin{array}{rcl}
\gamma(q_h, i_{Q_h}(\u_h\cdot\bar\u_h))&\le&\|\partial_t\u_h\|_h \|\bar\u_h\|_h +\gamma \|\nabla\u_h\| \| \nabla\bar\u_h\|+\alpha \|\partial_t\u_h\|_{h} \|\u_h\|_{h} \|\bar\u_h\|_{\boldsymbol{L}^\infty(\Omega)} 
\\
&\le &  \left((1+\alpha)\sqrt{\meas{\Omega}} \|\partial_t \u_h\|_h +\gamma\|\nabla\u_h\|\right) \left(\|\bar\u_h\|_{\boldsymbol{L}^\infty(\Omega)}+\|\nabla\bar\u_h\|\right).
\end{array}
$$
Therefore,
$$
\frac{(q_h, i_{Q_h}(\u_h\cdot \bar\u_h))}{\|\bar\u_h\|_{\boldsymbol{L}^\infty(\Omega)}+\|\nabla\bar\u_h\|}\le \frac{(1+\alpha)}{\gamma} \sqrt{\meas{\Omega}}  \|\partial_t \u_h\|_h  +\|\nabla\u_h\|.
$$
Applying (\ref{inf-sup}) above, we find
$$
\|q_h\|_{(H^1(\Omega)\cap L^\infty(\Omega))'} \le C (1+\|\nabla\u_h\|) \left(\frac{(1+\alpha)}{\gamma}\sqrt{\meas{\Omega}}\|\partial_t \u_h\|_h  +\|\nabla\u_h\|\right).
$$ 
The proof follows by using (\ref{bound_energy_for_uh}).
\end{proof}

\begin{remark} Observe that if we apply directly to $(\ref{FEM})_1$ the argument leading to (\ref{inf-sup})  so as to  obtain an estimate for $q_h$ we will improve estimate (\ref{est:q_hinL2}) in time. Indeed, let $q\in H^1(\Omega)\cap L^\infty(\Omega)$ and  select $\bar\u_h=\boldsymbol{i}_{\U_h}(\u_h\pi_{Q_h}(q))$ in $(\ref{FEM})_1$. Then we find
$$
\gamma(q_h, q)=(\partial_t \u_h, \u_h\cdot \boldsymbol{i}_{\U_h}(\u_h\pi_{Q_h}(q)))+\gamma(\nabla\u_h, \nabla \boldsymbol{i}_{\U_h}(\u_h\pi_{Q_h}(q)))+\alpha(\u_h\times\partial_t\u_h, \boldsymbol{i}_{\U_h}(\u_h\pi_{Q_h}(q))).
$$ 
Noting that both $(\partial_t \u_h, \u_h\cdot \boldsymbol{i}_{\U_h}(\u_h\pi_{Q_h}(q)))=0$ and $(\u_h\times\partial_t\u_h, \boldsymbol{i}_{\U_h}(\u_h\pi_{Q_h}(q)))=0$, we obtain, by (\ref{lm4.6-lab1}), that   
$$
(q_h, q)\le C \|\nabla\u_h\| (1+\|\nabla\u_h\|) (\|q\|_{L^\infty(\Omega)}+\|\nabla q_h\|).
$$
Therefore,
\begin{equation}\label{est:q_hinLinf}
\|q_h\|_{L^\infty(0,+\infty; (H^1(\Omega)\cap L^\infty(\Omega))')}	\le C \|\nabla\u_0\| (1+\|\nabla\u_0\|).
\end{equation}
\end{remark}

\begin{remark} Replacing $(q_h, i_{Q_h}(\u_h\cdot\bar\u_h))$  with $(q_h, \u_h\cdot\bar\u_h)_h$  in (\ref{FEM}), we obtain the following scheme. Find $(\u_h, q_h)\in C^\infty([0,+\infty);\U_h)\times C^\infty([0,+\infty); Q_h)$ such that, for all $(\u_h, q_h)\in \U_h\times Q_h$,
\begin{equation}\label{FEM-h}
\left\{
\begin{array}{rcl}
(\partial_t\u_h, \bar\u_h)_h+ \gamma(\nabla\u_h, \nabla\bar\u_h)+\gamma(q_h, \u_h\cdot\bar\u_h)_h+\alpha(\u_h\times\partial_t\u_h, \bar\u_h)_h&=&0,
\\
(\u_h\cdot\u_h, \bar q_h )_h&=&(1, \bar q_h)_h.
\end{array}
\right.
\end{equation}

Then, Lemma \ref{lm:local_existence} holds, and the nodal enforcement (\ref{bound_Linfty_for_uh}) and the  energy law (\ref{bound_energy_for_uh}) are valid for scheme (\ref{FEM-h}). Moreover, the inf-sup condition (\ref{inf-sup}) can be proved by selecting $\bar\u_h=\boldsymbol{i}_{\U_h}(\u_h P_h(q))$ where $P_h$ is defined by
$$
(P_h(\u_h), \bar\u_h)_h=(\u_h, \bar\u_h)\quad\mbox{for all}\quad \bar\u_h\in \U_h.
$$ 
\end{remark}

\section{Temporal discretization}\label{sec5}
In this section we shall propose two time integrators for (\ref{FEM}) which preserve the energy law (\ref{bound_energy_for_uh}) and estimate (\ref{est:q_hinLinf}). More precisely, we will construct a linearly implicit Euler and a nonlinearly implicit Crank--Nicolson time-stepping algorithm. For the linear one, we will require an extra assumption on the mesh $\mathcal{T}_h$. 
\begin{enumerate}
\item [(H5)] Assume  $\mathcal{T}_h$ to satisfy that  if $\u_h \in \U_h$ with $|\u_h(\a)| \ge 1$ for all $\a \in\mathcal{ N}_h$, then
$$
\|\nabla \boldsymbol{i}_{\U_h}(\frac{\u_h}{|\u_h|})\|\le \|\nabla\u_h\|. 
$$
\end{enumerate}

Assumption $\rm( H5)$ is assured under the condition \cite{Bartels_2005} 
$$
\int_\Omega\nabla \varphi_\a\cdot\nabla \varphi_{\tilde\a}\le 0\quad \mbox{ for all }\quad \a,\tilde\a\in\mathcal{N}_h\quad\mbox{ with } \a\not =\tilde\a,
$$  
where we remember that $\{\varphi_\a: \a\in \mathcal{N}_h\}$ is the nodal basis of $X_h$. In particular, such a condition holds for meshes of the Delaunay type in two dimensions and with all dihedral angles of the tetrahedra being at most $\pi/2$ in three dimensions.

It is assumed here for simplicity that we have a uniform partition of $[0,T]$ into $N$ pieces. So, the time step size is
$k = T/N$ and the time values $(t_n = n k)_{n=0}^N$. To simplify the notation let us denote $\delta_t\u^{n+1}=\frac{\u^{n+1}-\u^n}{k}$.

First we present a first-order linear numerical scheme.

\begin{center}
\noindent\fbox{
\begin{minipage}{0.8\textwidth}
\textbf{Algorithm 1}: Euler time-stepping scheme
\end{minipage}
}

\noindent\fbox{
\begin{minipage}{0.8\textwidth}
\textbf {Step $(n+1)$}: Given $\u^{n}_h\in \U_{h}$, find $(\u^{n+1}_h, q^{n+1}_h )\in \U_{h}\times Q_h$ solving the algebraic linear system
\begin{equation}\label{FEM-Euler}
\left\{
\begin{array}{rcl}
\displaystyle
(\delta_t \u_h^{n+1}, \bar\u_h)_h+ \gamma (\nabla\u^{n+1}_h, \nabla\bar\u_h)&&
\\
\displaystyle
+\gamma (q^{n+1}_h,  i_{Q_h}(\frac{\u^{n}_h}{|\u^{n}_h|}\cdot \bar\u_h))
+\alpha (\u_h^n \times \delta_t\u^{n+1}_h, \bar\u_h)_h &=&0, 
\\
( i_{Q_h} (\u^{n}_h \cdot \delta_t\u^{n+1}_h), \bar q_h)&=&0,
\end{array}
\right.
\end{equation}
for all $(\bar\u_h, \bar q_h,) \in \U_h\times Q_h$.

\end{minipage}
}

\end{center}

\begin{theorem}\label{Th:Euler} Assume that assumptions $\rm(H1)$--$\rm(H5)$ hold. Let  $\{\u_h^m\}_{m=1}^N$ be the numerical solution of~(\ref{FEM-Euler}). Then 
\begin{equation}\label{Energy-Euler}
\sum_{n=0}^m k \left(\|\delta_t \u^{n+1}_h\|_h^2 + \frac{\gamma k}{2} \|\nabla \delta_t \u^{n+1}_h\|^2\right) + \frac{\gamma}{2}\|\nabla\u^{m+1}_h\|^2=\frac{\gamma}{2}\|\nabla \u^0_h\|^2.
\end{equation}
Moreover, the Lagrange multiplier $\{q_h^m\}_{m=1}^N$  satisfies 
\begin{equation}\label{Inf-sup-Euler}
\max_{n=1, \cdots, N}\|q_h^{n}\|_{(H^1(\Omega)\cap L^\infty(\Omega))'}\le C (1+\|\nabla \u_h^0 \|) \|\nabla\u_h^0\|,
\end{equation}
where $C>0$ is a constante independent of $h$ and $k$. 
\end{theorem}

\begin{proof} Let $\bar\u_h= 2\, k\, \delta_t\u^{n+1}_h$ in $(\ref{FEM-Euler})_1$ to get 
$$
2\,k\|\delta_t\u^{n+1}_h\|_h^2+ \gamma\|\nabla \u^{n+1}_h\|^2-\gamma\|\nabla \u^n_h\|^2+\gamma k^2\|\nabla\delta_t\u^{n+1}_h\|^2+2\gamma k(q_h^{n+1}, i_{Q_h}(\frac{\u^n}{|\u^n_h|}\cdot\delta_t \u^{n+1}_h))=0,
$$
where the damping term has disappeared. From $(\ref{FEM-Euler})_2$, we infer that $ \u^n_{\a} \cdot \delta_t\u^{n+1}_{\a}=0 $ for all $\a\in\mathcal{N}_h$. Therefore the last term in the above equation vanishes.  Thus, it follows that (\ref{Energy-Euler}) holds by summing over $n$.

To prove  the inf-sup condition, we select $\bar\u_h=\boldsymbol{i}_{\U_h}(\frac{\u^{n}_h}{|\u^{n}_h|}\pi_{Q_h}(q))$ in $(\ref{FEM-Euler})_1$, with $q\in H^1(\Omega)\cap L^\infty(\Omega)$, to obtain 
$$
(q_h^{n+1},  q)= (\nabla\u^{n+1}_h, \nabla (\boldsymbol{i}_{\U_h}(\frac{\u^{n}_h}{|\u^{n}_h|}\pi_{Q_h}( q))))=(\nabla\u^{n+1}_h, \nabla (\boldsymbol{i}_{\U_h}(\boldsymbol{i}_{\U_h}(\frac{\u^{n}_h}{|\u^{n}_h|})\pi_{Q_h}( q)))).
$$
Using estimate  (\ref{est-H1}), we have
$$
(q^{n+1}_h,  q)\le C \|\nabla \u^{n+1}_h\|  \left( \|\nabla\boldsymbol{i}_{\U_h}(\frac{\u^{n}_h}{|\u^n_h|})\| \| q\|_{L^\infty(\Omega)} + \|\nabla  q\| \|\boldsymbol{i}_{\U_h}(\frac{\u^n_h}{|\u^n_h|})\|_{\boldsymbol{L}^\infty(\Omega)} \right).
$$
In view of $(\ref{FEM-Euler})_2$, we deduce that 
$$
0=\u^n_{\a}\cdot\delta_t\u^{n+1}_{\a}=|\u^{n+1}_{\a}|^2-|\u^{n}_{\a}|^2-|\u^{n+1}_{\a}-\u^{n}_{\a}|;
$$
hence $|\u^n_{\a}|\ge |\u^{n-1}_{\a}|\ge1$ holds since $|\u^0_{\a}|=1$ for all $\a\in\mathcal{N}_h$. This fact combined with assumption $\rm( H5)$ yields
$$
\begin{array}{rcl}
(q^{n+1}_h,  q)&\le& C \|\nabla \u^{n+1}_h\|  ( \|\nabla\u^n_h\| \| q\|_{L^\infty(\Omega)} + \|\nabla  q\|  )
\\
&\le& C \|\nabla \u^{n+1}_h\|  (1+ \|\nabla\u^n_h\|) (\| q\|_{L^\infty(\Omega)} + \|\nabla q\|).
\end{array}
$$
Estimate (\ref{Inf-sup-Euler}) then follows by utilizing duality and (\ref{Energy-Euler}).
\end{proof}
Equation~(\ref{Energy-Euler}) is the fully discrete counterpart of~(\ref{energy}) and~(\ref{bound_energy_for_uh}).

Next we deal with a second-order approximation based on a Crank--Nicolson method.
\begin{center}
\noindent\fbox{
\begin{minipage}{0.8\textwidth}
\textbf{Algorithm 2}: Crank--Nicolson time-stepping scheme
\end{minipage}
}
\noindent\fbox{
\begin{minipage}{0.8\textwidth}
\noindent {\bf Step $(n+1)$:} Given $\u^{n}_h\in \U_{h}$, find $(\u^{n+1}_h, q^{n+1}_h )\in \U_{h}\times Q_h$
solving the algebraic nonlinear system
\begin{equation}\label{FEM-CN}
\left\{
\begin{array}{rcl}
\displaystyle
(\delta_t \u_h^{n+1}, \bar\u_h)_h+\gamma (\nabla\u^{n+\frac{1}{2}}_h, \nabla\bar\u_h)&&
\\
\displaystyle
+\gamma(q^{n+\frac{1}{2}}_h,  i_{Q_h}(\frac{\u^{n+\frac{1}{2}}_h}{|\u^{n+\frac{1}{2}}_h|}\cdot \bar\u_h))
+\alpha (\u_h^n\times\delta_t\u_h^{n+1}, \bar\u_h)&=&0, 
\\
( i_{Q_h} (\u^{n+1}_h \cdot \u^{n+1}_h), \bar q_h)&=&(1, \bar q_h ),
\end{array}
\right.
\end{equation}
for all $(\bar\u_h, \bar q_h,) \in \U_h\times Q_h$.  
\end{minipage}
}

\end{center}
\begin{theorem} Assume that assumptions $\rm(H1)$--$\rm(H4)$ are satisfied. Let  $\{\u_h^m\}_{m=1}^N$ be the numerical solution of~(\ref{FEM-CN}). Then
\begin{equation}\label{Energy-CN}
        \sum_{n=0}^m k\|\delta_t \u^{n+1}_h\|^2_h + \frac{\gamma}{2}\|\nabla\u^{m+1}_h\|^2=\frac{\gamma}{2}\|\nabla \u^0_h\|^2.
\end{equation}
Moreover, the Lagrange multiplier $\{q_h^m\}_{m=1}^N$ satisfies
\begin{equation}\label{Inf-sup-CN}
\max_{n=1, \cdots N}\|q_h^{n}\|_{(H^1(\Omega)\cap L^\infty(\Omega))'}\le  C (1+\|\nabla \u_h^0 \|) \|\nabla\u_h^0\|,
\end{equation}
where $C>0$ is a constante independent of $h$ and $k$. 
\end{theorem}
\begin{proof} As in the proof of Theorem (\ref{Th:Euler}), we substitute  $\bar\u_h=\delta_t\u_h^{n+1}$ into $(\ref{FEM-CN})_1$ and $\bar q_h=q^{n+1}_h$ into $(\ref{FEM-CN})_2$ to obtain (\ref{Energy-CN}), and then $\bar\u_h=\boldsymbol{i}_{\U_h}(\u_h^{n+\frac{1}{2}}\pi_{Q_h}(\bar q))$ into $(\ref{FEM-CN})_1$ to get (\ref{Inf-sup-CN}). 
\end{proof}

\begin{remark} In the next section we will use scheme (\ref{FEM-Euler}) as a non-linear solver for approximating each step of scheme (\ref{FEM-CN}) when rewritten in the appropriate fashion. 
\end{remark}

\section{Implementation details}
\label{sect:implementation-details}

The second order time integrator~(\ref{FEM-CN}) requires solving at
each time step a nonlinear system. In this section, we discuss a
possible solution strategy for such a problem, considering for
simplicity the case $\alpha=0$. The first step is
rewriting~(\ref{FEM-CN}) in terms of $\w_h=\u^{n+\frac{1}{2}}_h$ and
$s_h = q^{n+\frac{1}{2}}_h$ as
\begin{equation}
\left\{
 \begin{array}{rcl}
  \displaystyle
   ( \w_h-\u_h^n , \bar\u_h )_h
+ \frac{\gamma k}{2} ( \nabla\w_h,\nabla\bar\u_h )
+ \frac{\gamma k}{2} ( s_h ,i_{Q_h}\left( \frac{\w_h}{|\w_h|}
\cdot \bar\u_h \right) ) & = & 0
\\[4mm] \displaystyle
( i_{Q_h}\left( \w_h\cdot(\w_h-\u_h^n) -\frac{1-\u_h^n\cdot\u_h^n}{4}
\right),\bar{q}_h ) & = & 0.
 \end{array}
 \right.
\label{eq:imp-details-w}
\end{equation}
In~(\ref{eq:imp-details-w})$_2$, the term involving
$1-\u_h^n\cdot\u_h^n$ should vanish, thanks to the unit sphere
constraint. However, since the nonlinear problem in general can not be
solved exactly, we have two options: normalize $\u_h^n$ after each
time step, or accept a (small) violation of the unit sphere constraint
and include such a term. Notice that~(\ref{eq:imp-details-w}) is an
implicit Euler step for the solution at half time levels.

Observe now that~(\ref{eq:imp-details-w})$_2$ amounts to requiring that
the argument of $i_{Q_h}$ vanishes at each node of the triangulation;
assuming $\w_h\ne \boldsymbol{0}$ we can reformulate this constraint as
\begin{equation}
\frac{\gamma k}{2}
( i_{Q_h}\left( \frac{\w_h}{|\w_h|}\cdot(\w_h-\u_h^n)
-\frac{1-\u_h^n\cdot\u_h^n}{4|\w_h|}
\right),\bar{q}_h ) = 0.
\label{eq:imp-details-mod}
\end{equation}
Equations~(\ref{eq:imp-details-w})$_1$ and~(\ref{eq:imp-details-mod})
are now taken as the basis for a fixed point iteration: given
$\w_h^{(i)}$, let
\[
\p = \frac{\gamma k}{2} \frac{ \w_h^{(i)} }{|\w_h^{(i)}|}, \qquad
\Gamma^n = \frac{\gamma k}{2} \frac{1-|\u_h^n|^2}{4|\w_h^{(i)}|}
\]
and compute the next iteration solving
\begin{equation}
\left\{
 \begin{array}{rcl}
  \displaystyle
  ( \w^{(i+1)}_h , \bar\u_h )
  + \frac{\gamma k}{2} ( \nabla\w^{(i+1)}_h,\nabla\bar\u_h )
  + ( s^{(i+1)}_h ,i_{Q_h}\left( 
     \p
     \cdot \bar\u_h \right) ) & = & \displaystyle
     ( \u_h^n , \bar\u_h ), \\
  \displaystyle
  ( i_{Q_h}\left( \p\cdot\w^{(i+1)}_h \right),\bar{q}_h ) & =
&
 \displaystyle
 ( i_{Q_h}\left( \p\cdot\u_h^n +\Gamma^n\right),\bar{q}_h
).
 \end{array}
 \right.
 \label{eq:imp-details-fixed-point}
\end{equation}
Notice the analogy between~(\ref{eq:imp-details-fixed-point}) and the
linearly implicit method~(\ref{FEM-Euler}). The matrix of the linear
system~(\ref{eq:imp-details-fixed-point}) has a classical
\[
\left[
\begin{array}{cc}
A & B^T \\ B
\end{array}
\right]
\]
structure, where $A$ is symmetric and positive definite and is block
diagonal with each block corresponding to one spatial dimension.
Hence, it
can be solved either using a direct method or an iterative one, such
as the Uzawa algorithm~\cite{Elman_1994}, which would then naturally lead to a
Newton--Krylov approach for the original nonlinear
problem~(\ref{eq:imp-details-w}).

\section{Numerical results}
\label{sect:numerical-results}

We consider here some numerical experiments aiming at verifying
numerically the convergence of the proposed scheme as well as analyzing
its behaviour in presence of singular solutions, including the case of
singular solutions in two space dimensions, which is outside the scope
of the theory presented in this paper.

\subsection{Convergence test for smooth solutions} 

In two spatial dimensions, we can set $\u = [\cos\theta,\sin\theta]^T$
and observe that, for $\alpha=0$, (\ref{HFHM}) implies
$\partial_t\theta -\gamma\Delta\theta=0$ in $\Omega\times\R^+$ with
$\partial_\n\theta=0$ on $\partial\Omega\times \R^+$. This lets us
construct the following exact solution for $\Omega = (-1\,,1)^2$:
$$
\u=\left[
\begin{array}{c}
\cos\theta \\ \sin\theta
\end{array}
\right], \qquad
\theta = \Theta e^{ -\gamma(k_x^2+k_y^2)t }\cos(k_xx)\cos(k_yy),
\qquad
q = -(\partial_x\theta)^2 -(\partial_y\theta)^2,
$$
with $\Theta=\pi$, $\gamma=0.01$, $k_x=\pi$, $k_y=2\pi$. To verify the
convergence of the proposed discretization, we compare the numerical
results with the exact solution at $t=1$, using a collection of
structured triangular grids with $h=2^{-i}$, $i=1,\dots,8$ and the
Crank--Nicolson scheme (\ref{FEM-CN}) with time-step $k=0.1\cdot2^{-j}$,
for $j=0,\ldots,6$. In all the computations, the nonlinear iterations
are carried out until reaching convergence within machine precision,
which in practice amounts to performing $\mathcal{O}(10)$ nonlinear
iterations.

Since our results indicate that the error resulting from the time
discretization is smaller than the one resulting from the space
discretization for all the considered grid sizes and time-steps, we can
analyze the two effects separately, focusing first on the space
discretization error. In order to do this, we fix $k=1/640$,
corresponding to $j=6$, and collect the error norms for $\u$ and $q$ in
Tables \ref{tab:convergence-u} and \ref{tab:convergence-q},
respectively.

Concerning the $\|\cdot\|_{(H^1)'}$ norm appearing in
Table~\ref{tab:convergence-q} as well as in
Figure~\ref{fig:time-convergence}, it is computed as follows. First of
all, thanks to the Riestz theorem, given $g \in H^{-1}$ there is $r_g\in
H^1_0$ such that, for any $f\in H^1_0$, $\left< g , f
\right>_{(H^1)'\times H^1_0} = (r_g , f )_{H^1_0}$; moreover,
$\|g\|_{(H^1)'} = \|r_g\|_{H^1_0}$. The difficulty is that, taking
$g=q-q_h$, $r_{q-q_h} \notin Q_h$, so that we can not compute it. This
problem can be circumvented computing the $H^1_0$ projection of
$r_{q-q_h}$ on $Q_h$, denoted here as $\Pi r_{q-q_h}$, which is uniquely
determined by
\[
(\Pi r_{q-q_h},f_h)_{H^1_0} = 
(r_{q-q_h},f_h)_{H^1_0} = 
\left< q - q_h , f_h \right>_{(H^1)'\times H^1_0}
\]
for any $f_h\in Q_h$. As shown in~\cite[Th 5.8.3]{Brenner_Scott_2008},
$\|\Pi r_{q-q_h}\|_{H^1_0}$ provides a second order estimate in $h$ of
the desired norm.

\begin{table}
\caption{Computed error norms for $\u-\u_h$ for a collection of
structured triangular grids with $h=2^{-i}$, $i=1,\dots,8$. The
numerical convergence rates are also reported.}
\begin{center} 
\begin{tabular}{|c||l|c||l|c||l|c||l|c||} \hline
$i$ &
\multicolumn{2}{c||}{$\|\u-\u_h\|_{\boldsymbol{L}^1}$} &
\multicolumn{2}{c||}{$\|\u-\u_h\|_{\boldsymbol{L}^2}$} &
\multicolumn{2}{c||}{$\|\u-\u_h\|_{\boldsymbol{L}^\infty}$} &
\multicolumn{2}{c||}{$\|\u-\u_h\|_{\H^1}$} \\
\hline
\lower.3ex\hbox{1} & \lower.3ex\hbox{$4.8            $} &                 --   & \lower.3ex\hbox{$2.0            $}&                 --  &  \lower.3ex\hbox{$1.6            $}&                 --  &  \lower.3ex\hbox{$1.4\cdot10^{ 1}$}&                 --   \\
\lower.3ex\hbox{2} & \lower.3ex\hbox{$2.2            $} & \lower.3ex\hbox{1.1} & \lower.3ex\hbox{$1.0            $}& \lower.3ex\hbox{1.0}&  \lower.3ex\hbox{$8.6\cdot10^{-1}$}& \lower.3ex\hbox{0.9}&  \lower.3ex\hbox{$1.2\cdot10^{ 1}$}& \lower.3ex\hbox{0.3} \\
\lower.3ex\hbox{3} & \lower.3ex\hbox{$7.0\cdot10^{-1}$} & \lower.3ex\hbox{1.7} & \lower.3ex\hbox{$3.5\cdot10^{-1}$}& \lower.3ex\hbox{1.5}&  \lower.3ex\hbox{$3.7\cdot10^{-1}$}& \lower.3ex\hbox{1.2}&  \lower.3ex\hbox{$6.5            $}& \lower.3ex\hbox{0.8} \\
\lower.3ex\hbox{4} & \lower.3ex\hbox{$1.6\cdot10^{-1}$} & \lower.3ex\hbox{2.2} & \lower.3ex\hbox{$7.8\cdot10^{-2}$}& \lower.3ex\hbox{2.1}&  \lower.3ex\hbox{$9.0\cdot10^{-2}$}& \lower.3ex\hbox{2.0}&  \lower.3ex\hbox{$3.0            $}& \lower.3ex\hbox{1.1} \\
\lower.3ex\hbox{5} & \lower.3ex\hbox{$3.8\cdot10^{-2}$} & \lower.3ex\hbox{2.1} & \lower.3ex\hbox{$1.9\cdot10^{-2}$}& \lower.3ex\hbox{2.1}&  \lower.3ex\hbox{$2.0\cdot10^{-2}$}& \lower.3ex\hbox{2.1}&  \lower.3ex\hbox{$1.4            $}& \lower.3ex\hbox{1.1} \\
\lower.3ex\hbox{6} & \lower.3ex\hbox{$9.3\cdot10^{-3}$} & \lower.3ex\hbox{2.0} & \lower.3ex\hbox{$4.7\cdot10^{-3}$}& \lower.3ex\hbox{2.0}&  \lower.3ex\hbox{$5.2\cdot10^{-3}$}& \lower.3ex\hbox{2.0}&  \lower.3ex\hbox{$7.0\cdot10^{-1}$}& \lower.3ex\hbox{1.0} \\
\lower.3ex\hbox{7} & \lower.3ex\hbox{$2.3\cdot10^{-3}$} & \lower.3ex\hbox{2.0} & \lower.3ex\hbox{$1.2\cdot10^{-3}$}& \lower.3ex\hbox{2.0}&  \lower.3ex\hbox{$1.3\cdot10^{-3}$}& \lower.3ex\hbox{2.0}&  \lower.3ex\hbox{$3.5\cdot10^{-1}$}& \lower.3ex\hbox{1.0} \\
\lower.3ex\hbox{8} & \lower.3ex\hbox{$5.8\cdot10^{-4}$} & \lower.3ex\hbox{2.0} & \lower.3ex\hbox{$2.9\cdot10^{-4}$}& \lower.3ex\hbox{2.0}&  \lower.3ex\hbox{$3.2\cdot10^{-4}$}& \lower.3ex\hbox{2.0}&  \lower.3ex\hbox{$1.8\cdot10^{-1}$}& \lower.3ex\hbox{1.0} \\
\hline
\end{tabular}
\end{center} 
\label{tab:convergence-u}
\end{table}
\begin{table}
\caption{Computed error norms for $q-q_h$ for a collection of structured
triangular grids with $h=2^{-i}$, $i=1,\dots,8$. The negative Sobolev
norm error $\|q-q_h\|_{(H^1)'}$, for which the numerical convergence
rate is also reported, is estimated with $\|\Pi r_{q-q_h}\|_{H^1_0}$ as
discussed in the text. Notice that, since $q_h$ is naturally computed at
half time steps, the analytic solution is evaluated ad $t=1-k/2$.}
\begin{center} 
\begin{tabular}{|c|S[table-format=2.4]|c|S[table-format=3.4]|S[table-format=2.4]|S[table-format=3.4]|S[table-format=5.2]|} \hline
$i$ &
\multicolumn{2}{c|}{$\|q-q_h\|_{(H^1)'}$} &
\multicolumn{1}{c|}{$\|q-q_h\|_{L^1}$} &
\multicolumn{1}{c|}{$\|q-q_h\|_{L^2}$} &
\multicolumn{1}{c|}{$\|q-q_h\|_{L^\infty}$} &
\multicolumn{1}{c|}{$\|q-q_h\|_{H^1}$} \\
\hline
1 & 59.0784 &  --  & 131.3211 & 93.1235 & 109.1703 &   852.73 \\
2 &  9.1603 &  2.7 & 110.0432 & 72.2844 &  76.7221 &  1022.32 \\
3 & 16.4411 & -0.8 &  55.8291 & 35.7175 &  62.8248 &  1242.74 \\
4 &  4.8247 &  1.8 &  31.6080 & 23.3787 &  71.2639 &  2123.59 \\
5 &  1.2166 &  2.0 &  27.2561 & 20.9493 &  62.0149 &  3990.89 \\
6 &  0.3079 &  2.0 &  26.5922 & 20.5044 &  59.6844 &  7862.59 \\
7 &  0.0803 &  1.9 &  26.5218 & 20.4059 &  59.1055 & 15666.76 \\
8 &  0.0231 &  1.8 &  26.5223 & 20.3821 &  58.9610 & 31304.46 \\
\hline
\end{tabular}
\end{center} 
\label{tab:convergence-q}
\end{table}
For $\u$, second order convergence is observed in the $\boldsymbol{L}^1$, $\boldsymbol{L}^2$ and
$\boldsymbol{L}^\infty$ norms, while first order converge is observed in the $\H^1$
norm. For the Lagrange multiplier $q$, second order convergence is
observed in the $(H^1)'$ norm, while the $L^1$, $L^2$ and $L^\infty$
norms are bounded and the $H^1$ norm diverges. This behaviour of the
error for $q$ can be explained noting that the numerical approximation
exhibits grid scale oscillations maintaining a constant amplitude while
the grid is refined. It is important to stress, however, that such
oscillations are consistent with the stability estimates
\eqref{est:q_hinL2} and are not, thus, an
indication of numerical instability.

In order to isolate the error resulting from the time discretization, we
proceed by fixing the grid size $h$ and computing a reference solution
for a small time-step, which then allows computing the self convergence
rate. Taking as reference time-step $k^{\rm ref}=0.1\cdot2^{-8}$ we
observe, for all the considered grid sizes, second order convergence for
both $\u_h$ and $q_h$, for all the considered norms (indeed, for fixed
$h$ all these norms are equivalent);
\begin{figure}[ht] 
\begin{center}
\includegraphics[width=0.49\textwidth]{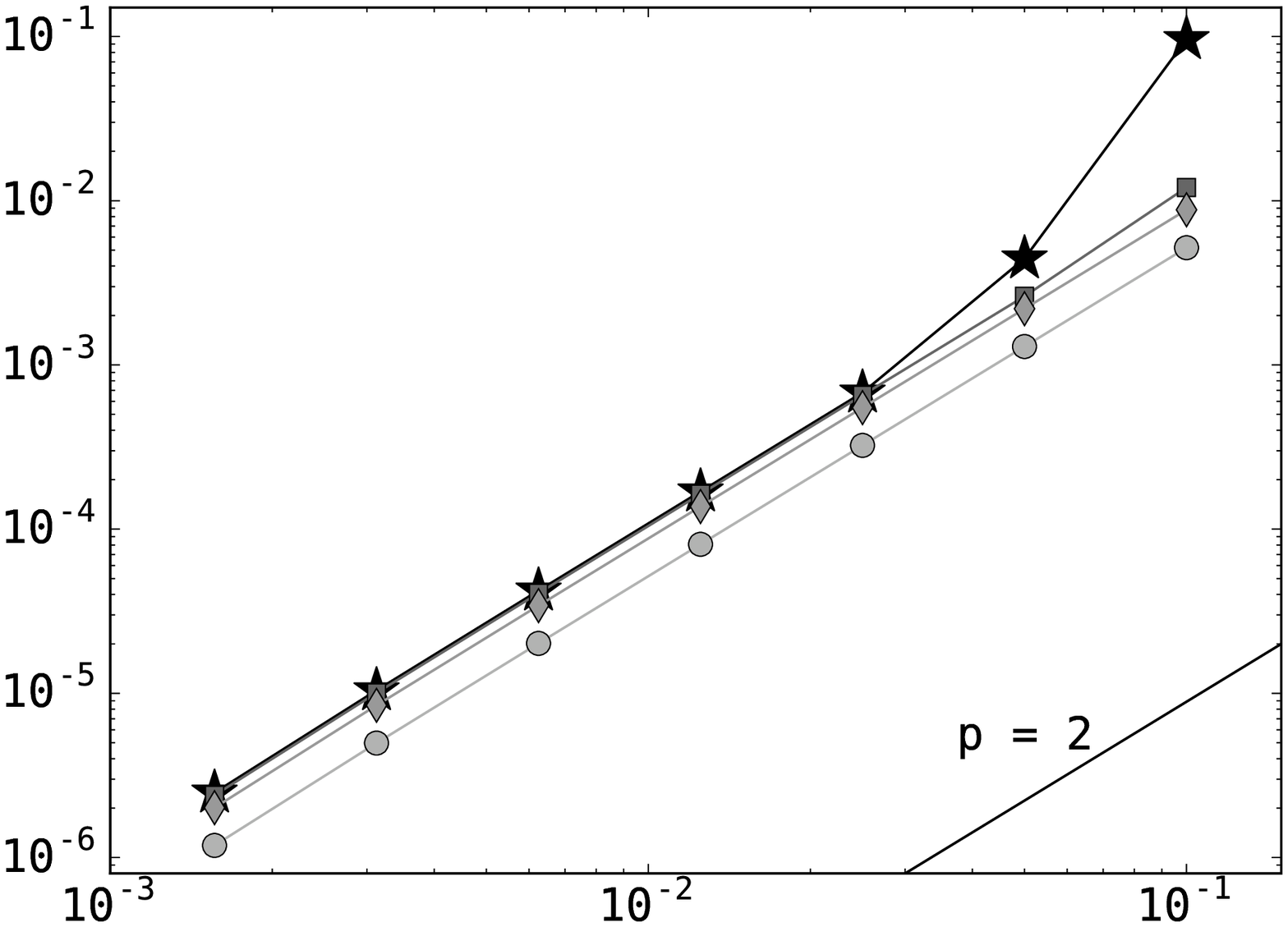}
\includegraphics[width=0.49\textwidth]{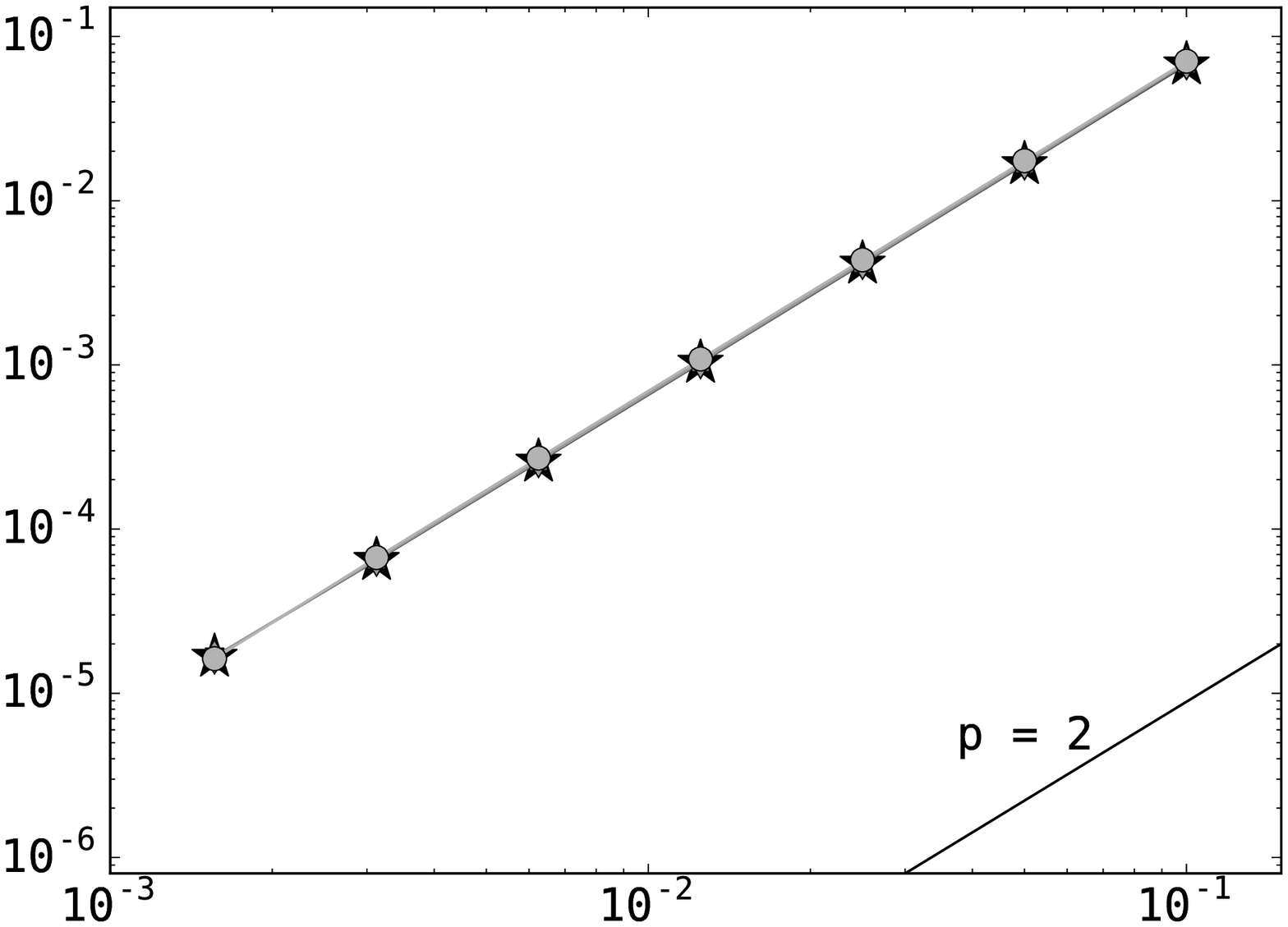}
\end{center}
\caption{Computed error norms $\|\u_h^{\rm ref}-\u_h\|_{\H^1}$ (left)
and $\|q_h^{\rm ref}-q_h\|_{(H^1)'}$ (right), for varying time-step
$k$, with $k^{\rm ref}=0.1\cdot2^{-8}$, and fixed mesh size. Results for
mesh sizes $h=2^{-3}$ ($\ocircle$), $h=2^{-4}$ ($\lozenge$), $h=2^{-5}$
($\square$) and $h=2^{-6}$ ($\bigstar$). Notice that, since $q_h$ is
naturally computed at half time steps, the corresponding errors are
computed at $t=1-k/2$ for each time-step $k$.}
\label{fig:time-convergence}
\end{figure}
Figure~\ref{fig:time-convergence} shows the results for
$\|\u-\u_h\|_{\H^1}$ and $\|q-q_h\|_{(H^1)'}$. A comparison of this
figure with the values reported in Tables \ref{tab:convergence-u} and
\ref{tab:convergence-q} confirms that the time discretization error is
smaller than the space discretization one. Second order convergence is
also apparent estimating the convergence rate as $\log_2\rho$, where
(see~\cite[Eq (4.7)]{Mu_2009})
\[
\rho =
\frac{\|\u_h^k-\u_h^{k/2}\|_{\H^1}}{\|\u_h^{k/2}-\u_h^{k/4}\|_{\H^1}},
\]
as shown in Table~\ref{tab:convergence-rho}.
\begin{table}
\caption{Estimated time discretization errors
$\|\u_h^k-\u_h^{k/2}\|_{H^1}$ for $k=0.1\cdot2^{-j}$, $j=0,\dots,5$, for
four triangular grids with $h=2^{-i}$, $i=3,4,5,6$ (see also
Figure~\ref{fig:time-convergence}). The resulting convergence rates are
also reported.}
\begin{center} 
\begin{tabular}{|c||l|c||l|c||l|c||l|c||} \hline
$j$ &
\multicolumn{2}{c||}{\lower.3ex\hbox{$h = 2^{-3}$}} &
\multicolumn{2}{c||}{\lower.3ex\hbox{$h = 2^{-4}$}} &
\multicolumn{2}{c||}{\lower.3ex\hbox{$h = 2^{-5}$}} &
\multicolumn{2}{c||}{\lower.3ex\hbox{$h = 2^{-6}$}} \\
\hline
\lower.3ex\hbox{0} & \lower.3ex\hbox{$3.9\cdot10^{-3}$} &                 --      & \lower.3ex\hbox{$6.6\cdot10^{-3}$}&                 --     &  \lower.3ex\hbox{$9.9\cdot10^{-3}$}&                 --     &  \lower.3ex\hbox{$9.2\cdot10^{-2}$}&                 --   \\
\lower.3ex\hbox{1} & \lower.3ex\hbox{$9.7\cdot10^{-4}$} & \lower.3ex\hbox{1.9989} & \lower.3ex\hbox{$1.6\cdot10^{-3}$}& \lower.3ex\hbox{2.0018}&  \lower.3ex\hbox{$2.0\cdot10^{-3}$}& \lower.3ex\hbox{2.3348}&  \lower.3ex\hbox{$4.1\cdot10^{-3}$}& \lower.3ex\hbox{4.5030} \\
\lower.3ex\hbox{2} & \lower.3ex\hbox{$2.4\cdot10^{-4}$} & \lower.3ex\hbox{1.9998} & \lower.3ex\hbox{$4.1\cdot10^{-4}$}& \lower.3ex\hbox{2.0004}&  \lower.3ex\hbox{$4.9\cdot10^{-4}$}& \lower.3ex\hbox{2.0004}&  \lower.3ex\hbox{$5.1\cdot10^{-4}$}& \lower.3ex\hbox{3.0037} \\
\lower.3ex\hbox{3} & \lower.3ex\hbox{$6.1\cdot10^{-5}$} & \lower.3ex\hbox{1.9999} & \lower.3ex\hbox{$1.0\cdot10^{-4}$}& \lower.3ex\hbox{2.0001}&  \lower.3ex\hbox{$1.2\cdot10^{-4}$}& \lower.3ex\hbox{2.0001}&  \lower.3ex\hbox{$1.3\cdot10^{-4}$}& \lower.3ex\hbox{2.0001} \\
\lower.3ex\hbox{4} & \lower.3ex\hbox{$1.5\cdot10^{-5}$} & \lower.3ex\hbox{2.0000} & \lower.3ex\hbox{$2.6\cdot10^{-5}$}& \lower.3ex\hbox{2.0000}&  \lower.3ex\hbox{$3.1\cdot10^{-5}$}& \lower.3ex\hbox{2.0000}&  \lower.3ex\hbox{$3.2\cdot10^{-5}$}& \lower.3ex\hbox{2.0000} \\
\lower.3ex\hbox{5} & \lower.3ex\hbox{$3.8\cdot10^{-6}$} & \lower.3ex\hbox{2.0000} & \lower.3ex\hbox{$6.4\cdot10^{-6}$}& \lower.3ex\hbox{2.0000}&  \lower.3ex\hbox{$7.7\cdot10^{-6}$}& \lower.3ex\hbox{2.0000}&  \lower.3ex\hbox{$7.9\cdot10^{-6}$}& \lower.3ex\hbox{2.0000} \\
\hline
\end{tabular}
\end{center} 
\label{tab:convergence-rho}
\end{table}

\subsection{Behaviour for singular solutions} 

After considering the behaviour of the scheme for problems with smooth
solutions, we turn our attention to problems including singularities. In
fact, singularities of the form $\u \sim \frac{\x-\x_0}{|\x-\x_0|}$
are very important in the study of liquid crystals, and various related
test cases have been considered in the
literature~\cite{Liu_2000,Du_2000, Liu_2002,Becker-Feng-Prohl_2008,Badia_Guillen_Gutierrez_JCP, Guillen_2015, Cabrales-Guillen-Gutierrez_2015}.

An important distinction here must be done between two- and
three-dimensional problems, since such singularities have a finite
energy in the three-dimensional case but not in the two-dimensional one
(mathematically, they belong to $\H^1(\Omega)$ for $\Omega\subset\R^3$
but not for $\Omega\subset\R^2$). From a practical perspective, a
singular solution can be approximated in the chosen finite element space
both in two and three spatial dimensions, for instance by nodal
interpolation (provided that none of the nodes coincides with the
singular point $\x_0$); the resulting function belongs to $\H^1$ and can
serve as an initial condition for a time dependent computation. Hence, one
might be tempted to dismiss the distinction between the two cases as a
merely theoretical argument with no practical implications. However,
this would be incorrect, as the following results demonstrate. Indeed,
for three-dimensional problems, the theoretical analysis provided above
holds, and our method is guaranteed to satisfy our stability
estimates when the grid is refined. For the two-dimensional case, on the
contrary, the theoretical analysis does not apply and nothing can be
said \emph{a priori}; nevertheless, computational experiments indicate
that, although it \emph{is} possible to compute a numerical solution,
such a solution critically depends on the numerical discretization, does
not converge to a well defined limit when the grid is refined and is
thus essentially meaningless.

Before discussing the numerical results, it is useful to provide a
qualitative analysis of the problem of representing a singular solution
with a finite element function. In two spatial dimensions, for a
structured, triangular grid, the finite element function will resemble
the patterns shown in Figure~\ref{fig:singularity-h}. This implies that
\begin{figure}[ht] 
\begin{center}
\includegraphics[width=0.75\textwidth]{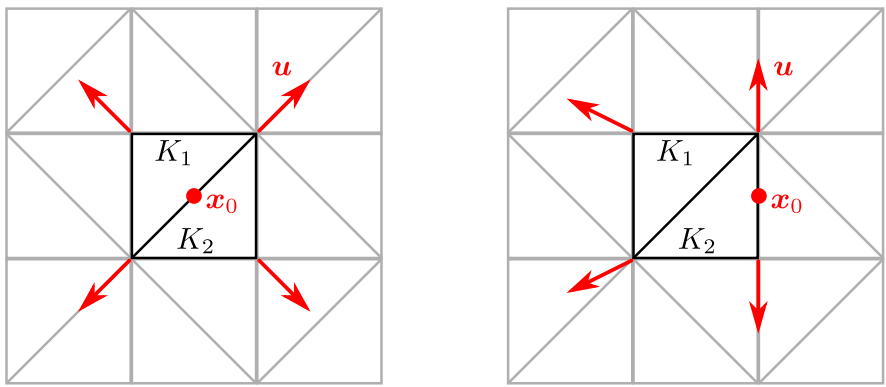}
\end{center}
\caption{Interpolation of a singular solution of the form
$\frac{\x-\x_0}{|\x-\x_0|}$ on a structured, triangular grid for
different location of $\x_0$ with respect to the computational grid.}
\label{fig:singularity-h}
\end{figure}
there are at least two elements where the numerical solution has
$\mathcal{O}(1)$ variations within an $\mathcal{O}(h)$ distance, which,
in turn, implies that the energy of the discrete solution
$\|\nabla\u_h\|_{\boldsymbol{L}^2}^2$ undergoes $\mathcal{O}(1)$ variations for an
$\mathcal{O}(h)$ displacement of the singularity. For instance,
referring to Figure~\ref{fig:singularity-h}, and given that we consider
linear finite elements, it easy to check that the two elements
containing the singularity, $K_1$ and $K_2$, contribute to the total
energy with $\|\nabla\u_h\|_{\boldsymbol{L}^2(K_1\cup K_2)}^2 = 4$ and
$\|\nabla\u_h\|_{\boldsymbol{L}^2(K_1\cup K_2)}^2 = (22-2\sqrt{5})/5$ for the two
depicted configurations, independently of $h$. Since the solution is
smooth far from the singularity, such $\mathcal{O}(1)$ energy variations
for $\mathcal{O}(h)$ displacements of the singularity are also present
if we consider the total energy $\|\nabla\u_h\|$; this is
shown in Figure~\ref{fig:singularity-h-energy} (left) 
\begin{figure}[ht] 
\begin{center}
\includegraphics[width=0.49\textwidth]{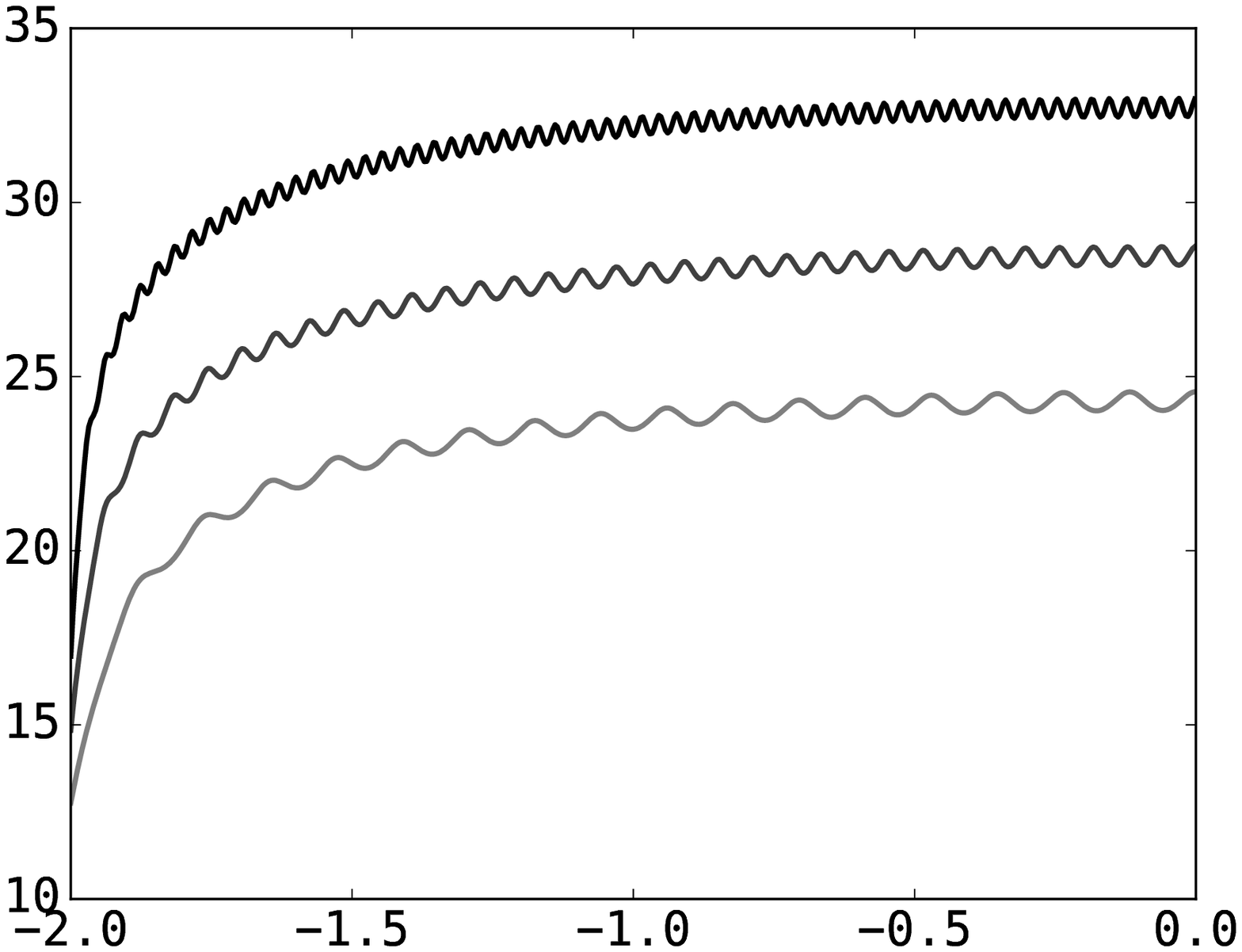}
\includegraphics[width=0.49\textwidth]{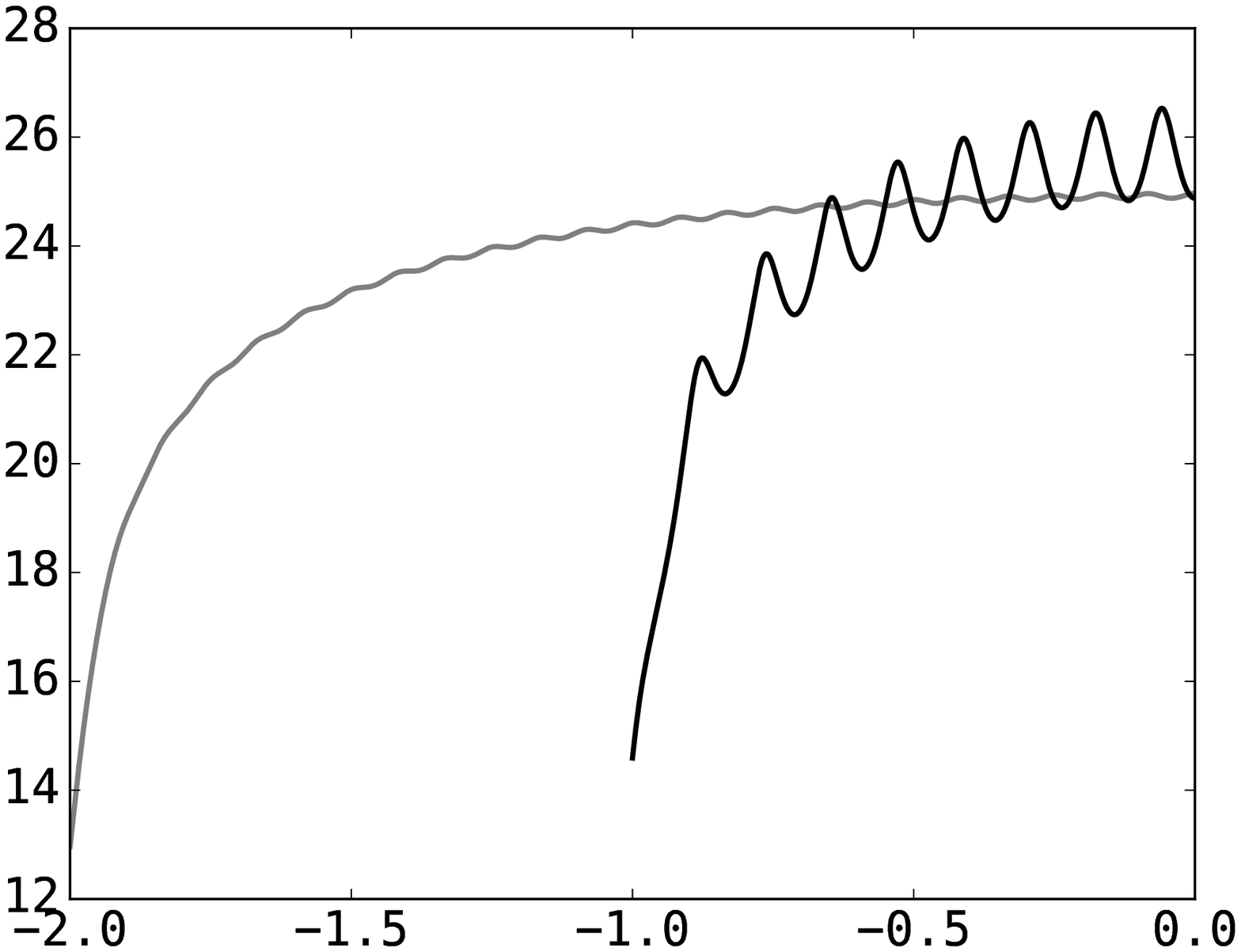}
\end{center}
\caption{Energy $\|\nabla\u_h\|^2$, for
$\Omega=(-2,2)\times(-1,1)$ and $\u_h$ defined as the nodal interpolant
of $\frac{\x-\x_0}{|\x-\x_0|}$ on a structured, triangular mesh, as a
function of $\x_0$. Left: values for three isotropic grids with
$34\times17$ elements (light gray), $66\times33$ elements (gray) and
$130\times65$ elements (black) as functions of $x_0$ such that
$\x_0=(x_0,0)^T$, $x_0\in[-2,0]$. Right: values for a single anisotropic
grid with $48\times17$ elements and $\x_0=(x_0,0)^T$ (gray) and
$\x_0=(1/24\,,y_0)^T$ (black), with $x_0\in[-2,0]$ and $y_0\in[-1,0]$.
The large-scale variations are due to the fact that, as $\x_0$
approaches the boundary of the domain, a ``large part'' of the field
lies outside $\Omega$.}
\label{fig:singularity-h-energy}
\end{figure}
where we plot, for various mesh sizes, the total energy of the nodal
interpolant of $\frac{\x-\x_0}{|\x-\x_0|}$ on
$\Omega=(-2,2)\times(-1,1)$ as a function of the position of $\x_0$,
specified as $\x_0=(x_0,0)^T$ for $x_0\in[-2,0]$. Since the energy can
not increase during the time evolution because of~(\ref{energy}), the
result of this grid dependence of the energy itself can be seen as a
``potential barrier'' which tends to trap the singularity between the
grid vertexes, in Figure~\ref{fig:singularity-h} (left). Two important characteristics of such a
barrier can be noted. First of all, it is independent of $h$, so
that refining the grid has no effect on it; this can be seen both by
noting that, for elements such as $K_1$ and $K_2$ in
Figure~\ref{fig:singularity-h}, $\|\nabla\u_h\|^2\sim h^{-2}$ and the
area element is proportional to $h^2$, as well as by considering
Figure~\ref{fig:singularity-h-energy} (left) where the amplitude of the
small-scale oscillations is constant for different resolutions. The
second characteristic of the potential barrier is its dependency on the
grid anisotropy. This is illustrated in
Figure~\ref{fig:singularity-h-aniso}, which shows how, for a uniform,
triangular grid with different spacings in the two Cartesian directions,
the variation of the finite element solution is more pronounced when the
singularity moves from $\x_0$ to $\x_0'$ compared to a displacement from
$\x_0$ to $\x_0''$.
\begin{figure}[ht] 
\begin{center}
\includegraphics[width=0.55\textwidth]{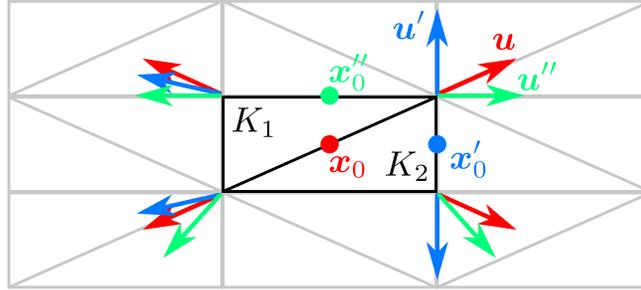}
\end{center}
\caption{Interpolation of a singular solution of the form
$\frac{\x-\x_0}{|\x-\x_0|}$ on a structured, triangular grid for
different location of $\x_0$ with respect to the computational grid.
Contrary to the case shown in Figure~\ref{fig:singularity-h}, we
consider here an anisotropic grid, with different spacings in the two
Cartesian directions.}
\label{fig:singularity-h-aniso}
\end{figure}
Again, this qualitative description is confirmed considering a specific
example of a grid composed of $48\times17$ elements for
$\Omega=(-2,2)\times(-1,1)$ and evaluating the energy of the finite
element solution when the singularity is displaced along the two
Cartesian axes, as shown in Figure~\ref{fig:singularity-h-energy}(right) where it is apparent that the amplitude of the grid-inducedoscillations is much larger when the singularity is displaced along the
direction with the largest grid spacing.

For three spatial dimensions, this qualitative analysis still holds, up
to one important difference: in such a case, $\mathcal{O}(1)$ variations
over an $\mathcal{O}(h)$ distance result in $\mathcal{O}(h)$ energy
contributions, because now it is still $\|\nabla\u_h\|^2\sim h^{-2}$ but
the area element is proportional to $h^3$. Hence, for three-dimensional
computations, the discrete potential barriers at the element boundaries
vanish when the grid is refined.

Summarizing now the conclusions of the qualitative analysis, we can
expect that, initializing the finite element computation interpolating a
singular field $\u$, the numerical solution will be strongly influenced
by the computational grid. Moreover, while three-dimensional computation
will converge to the analytic solution, two-dimensional ones will not
show any consistent limit when the grid is refined. Indeed, this is
precisely the outcome of our numerical experiments, which we now
describe in the remaining of the present section.

The initial condition for the numerical experiments is a modified
version of the one considered in~\cite{Liu_2002} and is defined as
\[
\u_0(\x) = \frac{\tilde{\u}_0(\x)}{|\tilde{\u}_0(\x)|}, \qquad
\tilde{\u}_0(\x) = w(\x)\,( \x + {\boldsymbol \delta} ) +
(1-w(\x))( -(\x - {\boldsymbol \delta}) )
\]
with
\[
w(\x) = \frac{1}{1+\exp(5x)}
\]
and $\x=(x,y)^T$, ${\boldsymbol \delta}=(\delta,0)^T$ and and
$\x=(x,y,z)^T$, ${\boldsymbol \delta}=(\delta,0,0)^T$ in two and three
space dimensions, respectively. This corresponds to two singularities
located on the $x$ axis approximately at $x=\pm\delta$ having opposite
sign and thus repelling each other. We take $\gamma=1$ and $\alpha=0$,
while the computational domain is $\Omega=(-2,2)\times(-1,1)$ in two
dimensions and $\Omega=(-2,2)\times(-1,1)\times(-1,1)$ in three
dimensions. The computational grid is uniform and structured and is
obtained, in two spatial dimensions, partitioning $\Omega$ into
rectangles with dimensions $\Delta_x,\Delta_y$ and dividing each
rectangle into two triangles with alternating direction, obtaining a
grid analogous to those depicted in Figures~\ref{fig:singularity-h}
and~\ref{fig:singularity-h-aniso}. For the three dimensional case, the
construction is similar witch each prism being divided into six
tetrahedral elements. Grids will be defined also by means of the number
of subdivisions in each Cartesian direction, so that a grid with
$N_x\times N_y$ elements correspond to $\Delta_x=4/N_x, \Delta_y=2/N_y$.
The overall evolution of the numerical solution is determined by the
interplay between the large-scale and the grid-scale energy variations
associated with a displacement of the two singularities: the former
corresponds to an energy decrease when the two singularities drift
apart, the latter has been analyzed previously in this section and tends
to lock the singularities between the grid vertexes. The initial
separation is chosen so that, for all the considered computations, a
transient is observed at least in the initial phase, with the
large-scale effect overcoming the grid one. In practice, we take
$\delta=0.0625$ in two dimensions and $\delta=0.5$ in three dimensions
and choose in both cases adequate grid anisotropy levels.

The time evolution of the energy of the finite element solution for the
two dimensional case in shown in Figure~\ref{fig:energy-time-2D-a} for
four levels of grid anisotropy and two levels of grid refinement.
\begin{figure}[ht] 
\begin{center}
\includegraphics[width=0.49\textwidth]{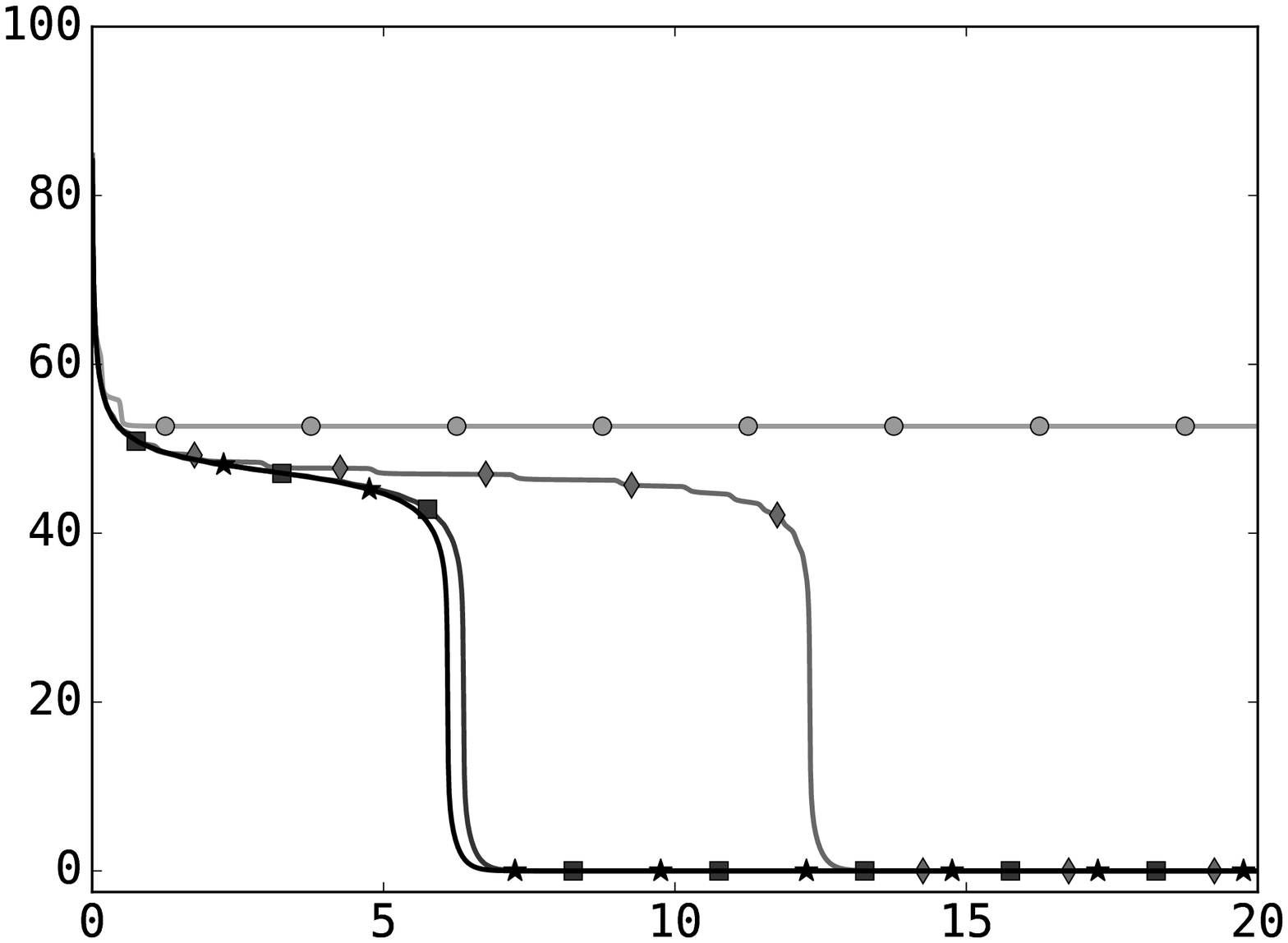}
\includegraphics[width=0.49\textwidth]{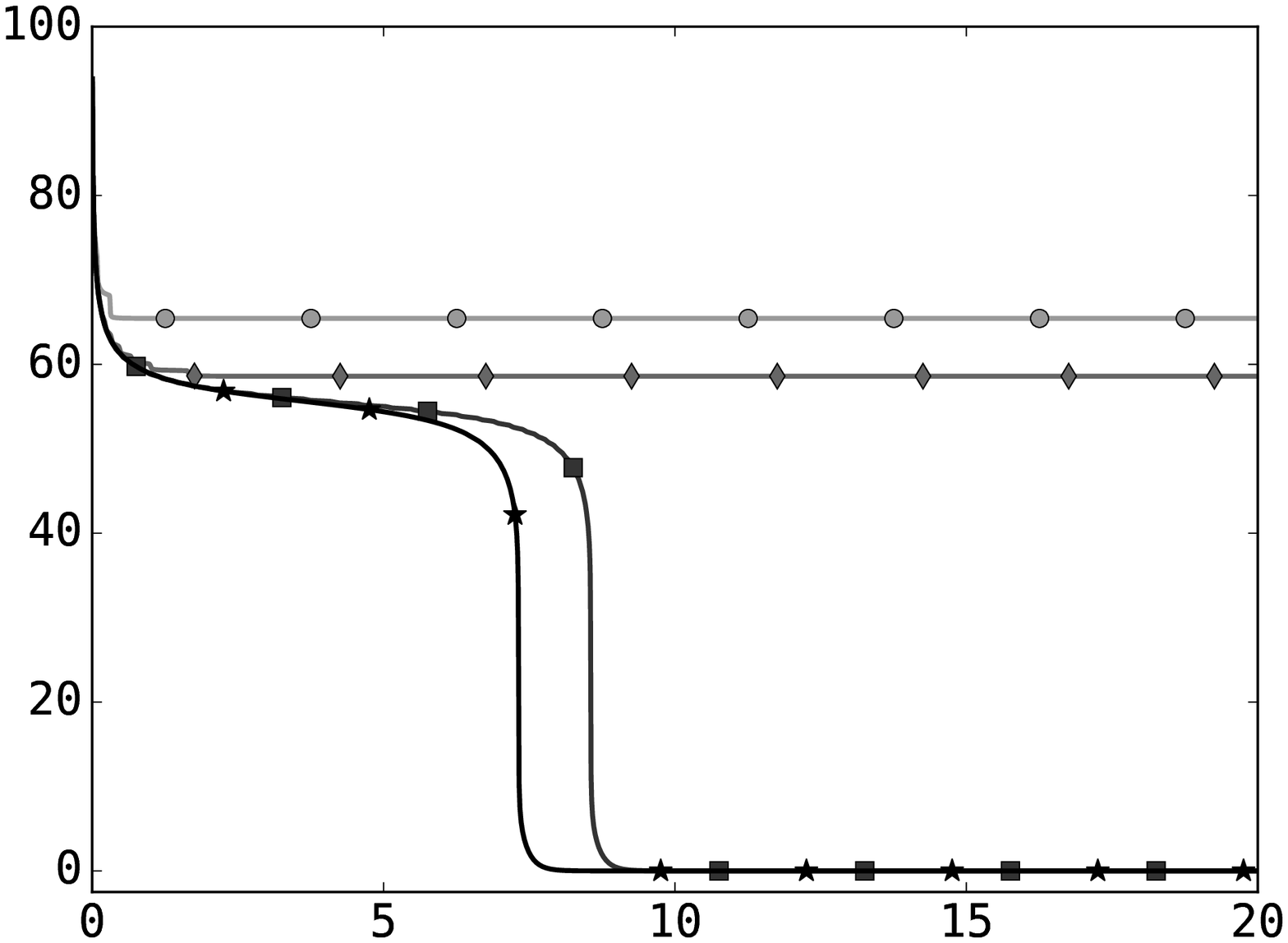}
\end{center}
\caption{Time evolution of $\|\nabla\u_h\|^2$ for two repelling
singularities in two space dimensions for different levels of grid
anisotropy, defined as $\Delta_y:\Delta_x$, namely
$1:1$ ($\ocircle$), $6:5$ ($\lozenge$), $7:5$ ($\square$) and $8:5$
($\bigstar$). Left: grids with $34\times17$, $41\times17$, $48\times17$
and $54\times17$ elements. Right: grids with $66\times33$, $79\times33$,
$92\times33$ and $106\times33$ elements.
}
\label{fig:energy-time-2D-a}
\end{figure}
The energy decreases as the two singularities drift apart and drops to
zero if they reach the boundary and leave the computational domain. The
first observation is that, depending on the anisotropy of the grid, the
two singularities can reach the boundary (when the anisotropy is such that
the potential barrier associated with the grid for displeacemet along
$x$ is small) or reach a steady state condition inside the grid after an
initial transient. The second observation is that the drift velocity is
strongly affected by the grid anisotropy. A third observation is that
the energy time evolution has a step pattern where each step
corresponds to the displacement of the singularities over one grid
element, i.e.\ to the crossing of one potential barrier. A fourth
observation is that, when the grid is refined, the effect of the grid is
not reduced: in fact, for finer grids, the spread among the computations
with similar resolution but different stretching increases and the
numerical steady state is reached earlier. Finally, we mention that that
analogous computations using unstructured grids, not reported here, show
that even the direction in which the singularities drift is strongly
affected by the computational grid.

Repeating now the experiment in three spatial dimensions, we obtain the
results reported in Figure~\ref{fig:energy-time-3D}. The step pattern
observed for two-dimensional computations is still present, however we
notice that: a) regardless of the grid anisotropy, the singularities
leave the computational domain; b) refining the grid, the amplitute of
the steps decreases and a tendency of the solutions obtained for
different levels of grid anisotropy to converge to a unique limit can be
observed.
\begin{figure}[ht] 
\begin{center}
\includegraphics[width=0.49\textwidth]{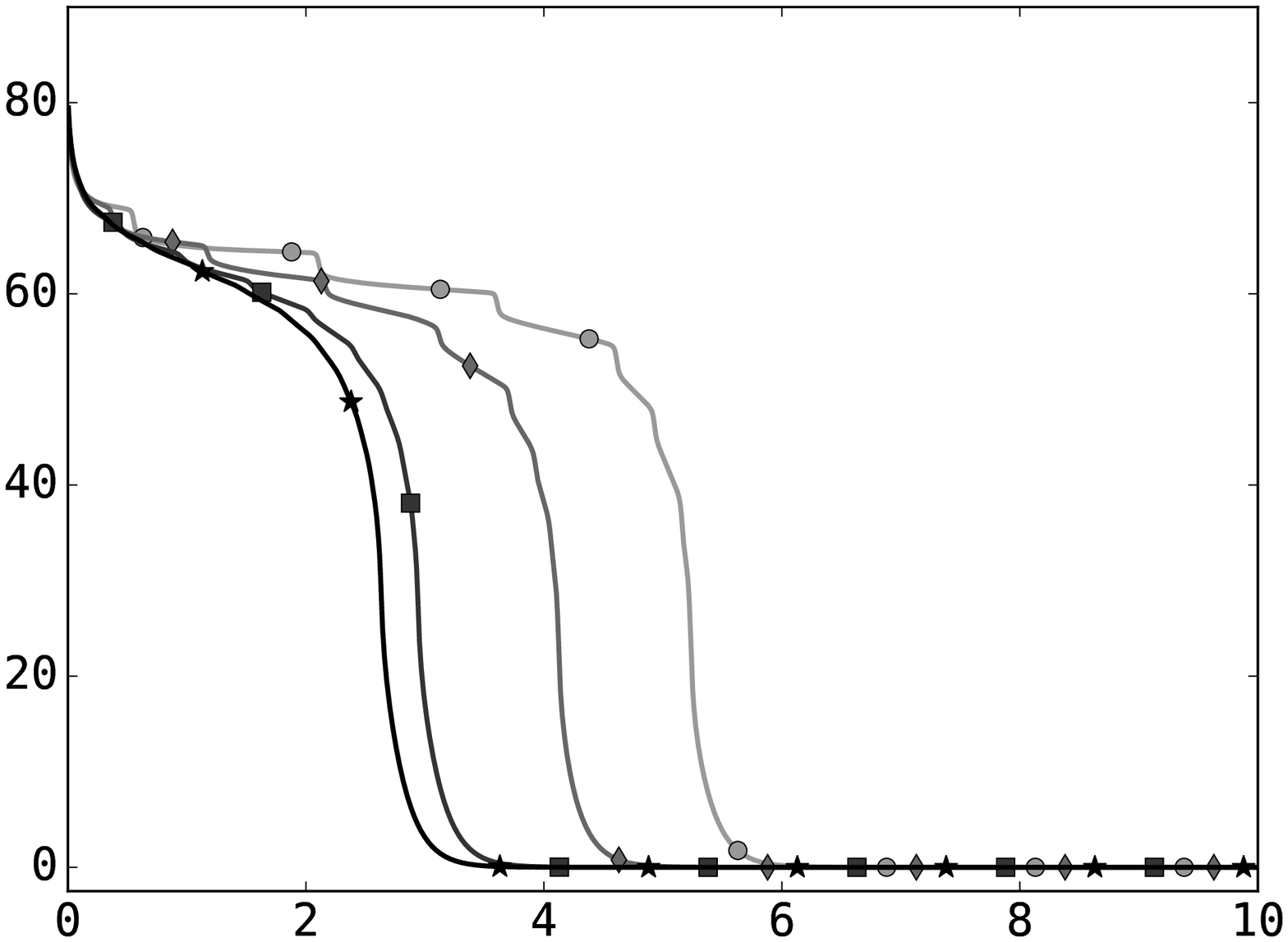}
\includegraphics[width=0.49\textwidth]{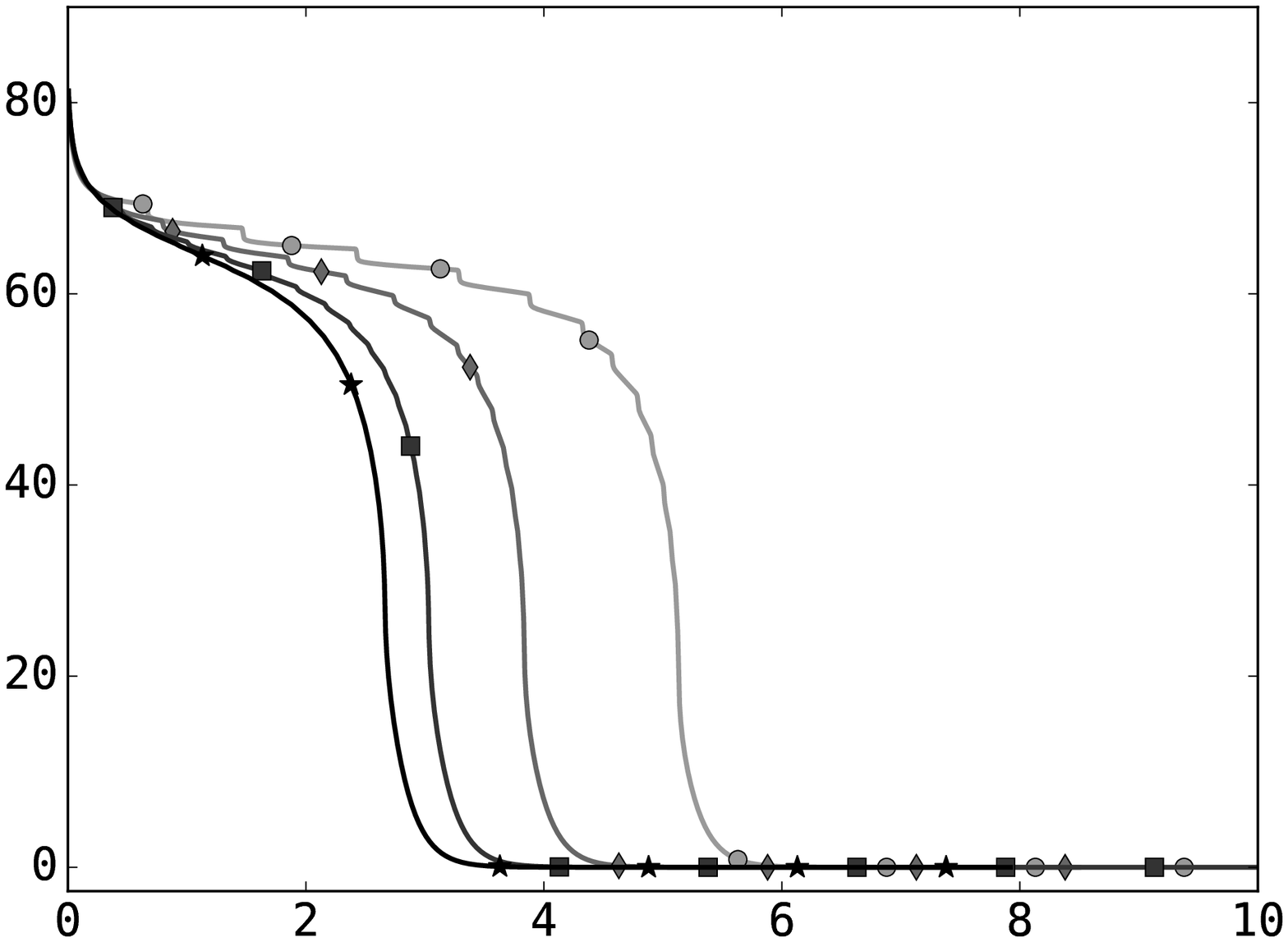}
\end{center}
\caption{Time evolution of $\|\nabla\u_h\|^2$ for two
repelling singularities in three space dimensions for different levels
of grid anisotropy, which can be defined as $\Delta_y:\Delta_x$ thanks
to the fact that $\Delta_y=\Delta_z$, namely: $10:17$ ($\ocircle$),
$10:15$ ($\lozenge$), $10:12$ ($\square$) and $1:1$ ($\bigstar$). Left:
grids with $20\times17\times17$, $23\times17\times17$,
$28\times17\times17$ and $34\times17\times17$ elements. Right: grids
with $38\times33\times33$, $44\times33\times33$, $53\times33\times33$
and $66\times33\times33$ elements.}
\label{fig:energy-time-3D}
\end{figure}

\section{Conclusion}
In this paper we have proposed and analyzed a unified saddle-point stable finite element method for approximating the harmonic map heat and Landau--Lifshitz--Gilbert equation. We have mainly proved that the numerical solution satisfies an energy law and a nodal satisfaction of the unit sphere, and the associated Lagrange multiplier satisfies an inf-sup condition. The key ingredients are using piecewise linear finite element spaces, applying a nodal interpolation to the terms involving the nonlinear restriction and a mass lumping technique to the terms involving time derivatives.    

This work has important implications in the context of the Ericksen--Leslie equations which incorporate a convective term to the harmonic map heat equation. While other existing approaches in the literature are not readily adapted to these equations due to the convective term, our proposed method may be directly applied to them without any modification keeping the desired properties above mentioned. 

Concerning the numerical results we have shown that the finite element solution computed by a nonlinear Crank--Nicolson method, which is solved by using semi-implicit Euler iterations, enjoys the expected accurate approximations. Moreover, we have identified that the dynamics of singularity points depends on dimension. That is, we have seen that, depending on the mesh anisotropy, in two dimensions, two singularity points can either be trapped among two elements of the mesh or move according to their sign. Instead, in three dimensions, the trapping effect does not occur. Therefore, some care must be taken in simulating singularities in two dimensions since these do not have finite energy and the limit equation only holds in the sense of measures.

\end{document}